\documentclass{amsart}
\usepackage{amsmath}
\usepackage{amsfonts}
\usepackage{graphics}
\usepackage{epsfig}
\usepackage{amssymb}
\usepackage{amscd}
\usepackage[all]{xy}
\usepackage{latexsym}
\usepackage{graphicx}
\usepackage{multirow}
\usepackage{geometry}
\usepackage{color}
\usepackage{hyperref}

\newtheorem{thm}{Theorem}[section]

\newtheorem{lem}[thm]{Lemma}
\newtheorem{prop}[thm]{Proposition}

\theoremstyle{definition}
\newtheorem{Def}[thm]{Definition}
\newtheorem{rem}[thm]{Remark}
\newtheorem*{ack}{Acknowledgement}

\numberwithin{equation}{section}
\numberwithin{figure}{section}

\def\rchi{{\hbox{\raise1.5pt\hbox{$\chi$}}}}

\def\tensor{\otimes}

\def\a{\alpha}
\def\b{\beta}
\def\lam{\lambda}
\def\Lam{\Lambda}
\def\gam{\gamma}

\def\Res{{\text{\rm{Res}}}}
\def\Sym{{\text{\rm{Sym}}}}

\def\tensor{\otimes}

\newcommand{\Mbar}{{\overline{\mathcal{M}}}}

\newcommand{\bP}{{\mathbb{P}}}
\newcommand{\bC}{{\mathbb{C}}}

\newcommand{\bE}{{\mathbb{E}}}

\newcommand{\bQ}{{\mathbb{Q}}}

\newcommand{\bZ}{{\mathbb{Z}}}

\newcommand{\cM}{{\mathcal{M}}}

\newcommand{\cD}{{\mathcal{D}}}

\newcommand{\cO}{{\mathcal{O}}}

\newcommand{\la}{{\langle}}
\newcommand{\ra}{{\rangle}}
\newcommand{\half}{{\frac{1}{2}}}

\newcommand{\rar}{\rightarrow}
\newcommand{\lrar}{\longrightarrow}

\newcommand{\Da}{\Delta}
\newcommand{\Db}{\widetilde{\Delta}}

\newcommand{\Dt}{\Delta}

\newcommand{\etatil}{{\widetilde\eta\,}}
\newcommand{\ytilde}{{\widetilde{y}\,}}

\newcommand{\ta}{\tau}
\newcommand{\lf}{\left}
\newcommand{\rt}{\right}
\newcommand{\pp}{\partial}
\renewcommand{\phi}{\varphi}

\newcommand{\lcb}{\lf\{}
\newcommand{\rcb}{\rt\}}

\newcommand{\vk}{\vec{k}}
\newcommand{\vd}{\vec{\ell}}

\newcommand{\kWithoutI}{k_{[\hat{i}]}}
\newcommand{\dWithoutI}{\ell_{[\hat{i}]}}
\newcommand{\etaWithoutI}{\eta_{[\hat{i}]}}

\newcommand{\kWithoutIJ}{k_{[\hat{i},\hat{j}]}}
\newcommand{\dWithoutIJ}{\ell_{[\hat{i},\hat{j}]}}
\newcommand{\etaWithoutIJ}{\eta_{[\hat{i},\hat{j}]}}

\newcommand{\kWithoutOne}{k_{[\hat{1}]}}
\newcommand{\dWithoutOne}{\ell_{[\hat{1}]}}
\newcommand{\etaWithoutOne}{\eta_{[\hat{1}]}}

\newcommand{\kWithoutOneJ}{k_{[\hat{1},\hat{j}]}}
\newcommand{\dWithoutOneJ}{\ell_{[\hat{1},\hat{j}]}}
\newcommand{\etaWithoutOneJ}{\eta_{[\hat{1},\hat{j}]}}

\begin{document}
\large
\setcounter{section}{0}

\title[Mirror symmetry for orbifold Hurwitz numbers]
{Mirror symmetry for orbifold Hurwitz numbers}

\author[V.\ Bouchard]{Vincent Bouchard}
\address{
Department of Mathematics and Statistical
Sciences\\
University of Alberta\\
Edmonton, AB, T6G 2G1, Canada}
\email{vincent@math.ualberta.ca}

\author[D.\ Hern\'andez Serrano]{Daniel Hern\'andez Serrano}
\address{
Department of Mathematics and IUFFYM\\
University of Salamanca\\
Salamanca 37008, Spain}
\email{dani@usal.es}

\author[X.\ Liu]{Xiaojun Liu}
\address{
Department of Applied Mathematics\\
China Agricultural University\\
Beijing, 100083, China\\
and
Department of Mathematics\\
University of California\\
Davis, CA 95616--8633, U.S.A.}
\email{xjliu@cau.edu.cn}

\author[M.\ Mulase]{Motohico Mulase}
\address{
Department of Mathematics\\
University of California\\
Davis, CA 95616--8633, U.S.A.}
\email{mulase@math.ucdavis.edu}

\begin{abstract}
We study mirror symmetry for orbifold Hurwitz numbers. We show that the Laplace transform of orbifold Hurwitz numbers satisfy a differential recursion, which is then proved to be equivalent to the integral recursion of Eynard and Orantin with spectral curve given by the $r$-Lambert curve. We argue that the $r$-Lambert curve also arises in the infinite framing limit of orbifold Gromov-Witten theory of $[\mathbb{C}^3/(\mathbb{Z}/r\mathbb{Z})]$. Finally, we prove that the mirror model to orbifold Hurwitz numbers admits a quantum curve.
\end{abstract}


\maketitle

\allowdisplaybreaks

\tableofcontents

\section{Introduction}
\label{sect:intro}

\subsection{Overview}

In recent years, it has been found that many counting problems involving the moduli space $\Mbar_{g,n}$,
such as Gromov-Witten invariants
of toric target spaces and enumeration of 
various ramified coverings of $\bP^1$, 
have a common feature:
they have a ``mirror symmetric'' counterpart which is governed by a universal integral recursion 
formula due to Eynard and Orantin \cite{EO1}.
The key ingredient of the mirror theory is the existence of a
\emph{spectral curve}, which is a Lagrangian 
subvariety of the holomorphic symplectic surface
$\bC^*\times \bC^*$. 
Once the spectral curve mirror to a given counting problem is determined,  
the universal recursion  
calculates the 
generating functions of the corresponding enumerative invariants. 

Simple Hurwitz numbers provide an interesting example of such a story. It was first conjectured in \cite{BM} that the generating functions for simple Hurwitz numbers should satisfy the integral recursion of Eynard and Orantin, with spectral curve given by the \emph{Lambert curve}
\begin{equation}
\label{eq:Lambert}
x=y \mathrm{e}^{-y}. 
\end{equation}
The conjecture followed from the broader \emph{remodeling conjecture} \cite{BKMP, M2}, which claims that generating functions for Gromov-Witten invariants of toric Calabi-Yau threefolds/orbifolds should satisfy the integral recursion of Eynard and Orantin, with spectral curve given by the standard mirror curve of Hori and Vafa \cite{HV}. The conjecture for simple Hurwitz numbers is derived as the infinite framing limit of the simplest case of the remodeling conjecture, namely for Gromov-Witten invariants of $\mathbb{C}^3$. 

The conjecture for simple Hurwitz numbers was solved in \cite{EMS,MZ}. There, it was shown that the generating functions of simple Hurwitz numbers defined in \cite{BM} are in fact the Laplace transform of the simple Hurwitz numbers $H_{g,n}(\vec{\mu})$ (defined below), and that the combinatorial equation known as the \emph{cut-and-join} equation \cite{Goulden,GJ, V}
automatically changes into the Eynard-Orantin integral recursion defined on the Lambert curve \eqref{eq:Lambert}, after taking the Laplace transform, Galois averaging, and restricting to the principal part. In this way the simple Hurwitz number conjecture was solved.

Through the infinite framing limit, the mathematical solution of the simple Hurwitz number conjecture presents a strong evidence for the remodeling conjecture itself. Recently, there have been many developments towards a proof of the remodeling conjecture (see for example \cite{BCMS,Chen, Zhou3}, and most notably,  \cite{EO3}). In its full generality, however,  the remodeling conjecture is still open. In particular, there are no rigorous mathematical results for the cases of orbifold Gromov-Witten invariants. 

In this paper we study mirror symmetry for Hurwitz numbers of the orbifold $\bP^1[r]$ with one stack point $\big[0\big/(\bZ/r\bZ)\big]$. 

As a first step, we use the remodeling conjecture to argue that the generating functions of such orbifold Hurwitz numbers should satisfy the integral recursion of Eynard and Orantin. As for simple Hurwitz numbers, we show that generating functions for orbifold Hurwitz numbers can be obtained in the infinite framing limit of generating functions for Gromov-Witten theory; however, instead of considering Gromov-Witten theory of $\mathbb{C}^3$, we must now consider orbifold Gromov-Witten theory of $[\mathbb{C}^3/ (\mathbb{Z} / r \mathbb{Z})]$. Via the remodeling conjecture, this implies that generating functions of orbifold Hurwitz numbers should satisfy the integral recursion, with spectral curve the infinite framing limit of the curve mirror to the orbifolds $[\mathbb{C}^3/ (\mathbb{Z} / r \mathbb{Z})]$. We show that the resulting spectral curve for orbifold Hurwitz numbers is the \emph{$r$-Lambert curve}:
\begin{equation}
\label{eq:rLamb}
x^r = y \mathrm{e}^{- r y}.
\end{equation}

We then give a rigorous proof of the recursion formula, generalizing the result of  \cite{BM, EMS, MZ} to the orbifold case. First, we prove that the $r$-Lambert curve is the correct spectral curve via Laplace transform. Then, we establish a system of recursive partial differential equations that uniquely determines the Laplace transform of the orbifold Hurwitz numbers for arbitrary genus and ramification profile at $\infty \in \bP^1[r]$. These functions are called \emph{free energies}. The Eynard-Orantin topological recursion is then established by taking the Galois average of the Laplace transform of the cut-and-join equation and restricting to the principal part of the free energies. Note that this result also provides strong evidence for the remodeling conjecture in the context of orbifold Gromov-Witten theory of $[\mathbb{C}^3/ (\mathbb{Z} / r \mathbb{Z})]$, which is still open.

We also study the appearance of a \emph{quantum curve} for orbifold Hurwitz numbers. Quantum curves arise when the mirror symmetric side of a counting problem is governed by a complex analytic curve. Here, a quantum curve \cite{ADKMV,DV2007,DHS,DHSV} means a \emph{holonomic system} that characterizes the \emph{partition function} of the theory, the latter being defined in terms of the principal specialization of the free energies. In the context of orbifold Hurwitz numbers, we show that the partition function (which is the diagonal restriction of a KP $\tau$-function) satisfies a stationary Schr\"odinger-type equation
of \cite{MSS}, that is, a quantum curve exists. Surprisingly, this linear equation alone uniquely determines the free energies for arbitrary genus. 

\subsection{Main results}

For a vector of $n$ positive integers
  $\vec{\mu}=(\mu_1,\dots,\mu_n)\in \bZ_+^n$,
the \emph{simple Hurwitz number}
 $H_{g,n}(\vec{\mu})$ counts the automorphism
weighted number of the topological types of 
simple Hurwitz covers of $\bP^1$ of type 
$(g,\vec{\mu})$. 
A holomorphic map $\varphi:C\lrar \bP^1$ is a
\emph{simple Hurwitz cover}
 of type $(g,\vec{\mu})$  
if $C$ is a complete nonsingular algebraic curve
defined
over $\bC$ of genus $g$, $\varphi$ has $n$ labeled 
poles of orders $(\mu_1,\dots,\mu_n)$, and all 
other critical 
points of $\varphi$ are unlabeled
simple ramification points.

In a similar way, we define the \emph{orbifold
Hurwitz number} $H_{g,n}^{(r)}(\vec{\mu})$
 for every positive integer
$r>0$ to be the automorphism weighted 
count of the topological types of 
smooth orbifold morphisms $\varphi:C\lrar \bP^1[r]$
with the same pole structure as the simple
Hurwitz number case. 
Here $C$ is a connected $1$-dimensional
orbifold (or a \emph{twisted} curve)
modeled on a nonsingular 
curve of genus $g$ with $(\mu_1+\cdots+\mu_n)/r$
stack points of the type $\big[p\big/(\bZ/r\bZ)\big]$.
We impose  that the inverse image of the
morphism 
$\varphi$ of the stack point 
$\big[0\big/(\bZ/r\bZ)\big]\in \bP^1[r]$
coincides with the set of stack points of $C$.
When $r=1$ we recover the 
simple Hurwitz number 
$H_{g,n}^{(1)}(\vec{\mu}) = 
H_{g,n}(\vec{\mu})$. 

Consider $H_{g,n}^{(r)}(\vec{\mu})$ as a function in $\vec{\mu}\in\bZ_+^n$. Following the recipe of \cite{DMSS, EMS, MS}, we define the \emph{free energies} as the Laplace transform
\begin{equation}
\label{eq:Fgn}
F^{(r)}_{g,n}(z_1,\dots,z_n) =
\sum_{\vec{\mu}\in \bZ_+^n}
H_{g,n}^{(r)}(\vec{\mu})\; e^{-\la \vec{w},\vec{\mu}\ra}.
\end{equation}
Here $\vec{w}=(w_1,\dots,w_n)$ is the 
vector of the Laplace
dual coordinates of $\vec{\mu}$, 
$\la \vec{w},\vec{\mu}
\ra=w_1\mu_1+\cdots +w_n\mu_n$,
and the function variable $z_i$ and $w_i$
for each $i$
are related by the $r$-\emph{Lambert}
function
\begin{equation}
\label{eq:r-Lambert}
e^{-w} = z e^{-z^r}.
\end{equation}
It is often convenient to use a different variable
\begin{equation}
\label{eq:x}
x= e^{-w},
\end{equation}
with which the $r$-Lambert curve is 
given by $x = ze^{-z^r}$. Then 
the free energies $F^{(r)}_{g,n}$ of (\ref{eq:Fgn})
are  generating functions of the 
orbifold Hurwitz numbers. We use the notation
\begin{equation}
\label{eq:Fgn in x}
F^{(r)}_{g,n}[x_1,\dots,x_n]
 = \sum_{\vec{\mu}\in \bZ_+^n}
H_{g,n}^{(r)}(\vec{\mu}) 
\prod_{i=1}^n x_i^{\mu_i}
\end{equation}
to indicate the same function (\ref{eq:Fgn}) in
the different set of variables. For 
every $(g,n)$ the power series (\ref{eq:Fgn in x})
is convergent and defines an analytic function.

Our first result, Theorem \ref{thm:limit}, states that the generating functions \eqref{eq:Fgn in x} can be obtained in the infinite framing limit of generating functions for orbifold Gromov-Witten invariants of $[\mathbb{C}^3 / (\mathbb{Z}/r \mathbb{Z})]$. This follows by rewriting both generating functions in terms of Hurwitz-Hodge integrals. On one side, a ELSV-type \cite{ELSV} formula expressing
orbifold Hurwitz numbers in terms of Hurwitz-Hodge integrals was established by Johnson-Pandharipande-Tseng \cite{JPT}, where orbifold Hurwitz  numbers were considered as a special case of double Hurwitz numbers. On the other side, orbifold Gromov-Witten invariants can also be expressed in terms of Hurwitz-Hodge integrals, through the orbifold topological vertex \cite{BC,Ross}. Using these expressions in terms of Hurwitz-Hodge integrals we establish the infinite framing correspondence for the generating functions.

Through the remodeling conjecture, it is expected that the generating functions for orbifold Gromov-Witten invariants of $[\mathbb{C}^3 / (\mathbb{Z}/r \mathbb{Z})]$ should satisfy the integral recursion of Eynard and Orantin with spectral curve
\begin{equation}
\label{eq:framedorbifold}
y^{s+rf}(1-y)-x^r=0,
\end{equation}
where $f \in \mathbb{Z}$ is a framing parameter and $s\in\bZ$ determines the weight of the action of $\bZ/r\bZ$ on $\bC^3$. By taking the limit of infinite framing, $f \to \infty$,
after an appropriate coordinate change
\begin{equation}
\begin{cases}
x\longmapsto \frac{x}{f^{1/r}}\\
y\longmapsto 1-\frac{y}{f}
\end{cases}
\end{equation}
we obtain the $r$-Lambert curve \eqref{eq:rLamb}. Therefore, we expect the free energies \eqref{eq:Fgn in x} to satisfy the integral recursion of Eynard and Orantin, with spectral curve the $r$-Lambert curve.

Our next result is an explicit determination of all the free energies \eqref{eq:Fgn}:
\begin{thm}
\label{thm:Fgn recursion}
In terms of the $z$-variables, the 
free energies are calculated as follows.
\begin{align}
\label{eq:F01}
F^{(r)}_{0,1}(z) &= \frac{1}{r}z^r-\half z^{2r},
\\
\label{eq:F02}
F^{(r)}_{0,2}(z_1,z_2) &=\log\frac{z_1-z_2}{x_1-x_2}
-(z_1^r+z_2^r),
\end{align}
where $x_i = z_i e^{-z_i ^r}$.
 For $(g,n)$ in the stable range, 
 i.e., when $2g-2+n>0$, the free energies 
 satisfy the differential recursion equation
 \begin{multline}
 \label{eq:diffrecursion}
 \left(
 2g-2+n+\frac{1}{r}\sum_{i=1}^n
 z_i \frac{\partial}{\partial z_i}
 \right)
 F^{(r)}_{g,n}(z_1,\dots,z_n) 
 \\
 =
 \half\sum_{i\ne j}\frac{z_iz_j}{z_i-z_j}
 \left[
 \frac{1}{(1-rz_i^r)^2}\frac{\partial}{\partial z_i}
 F^{(r)}_{g,n-1}\big(z_{[\hat{j}]}\big) -
  \frac{1}{(1-rz_j^r)^2}\frac{\partial}{\partial z_j}
 F^{(r)}_{g,n-1}\big(z_{[\hat{i}]}\big)
 \right]
 \\
 +
 \half \sum_{i=1}^n \frac{z_i^2}{(1-rz_i^r)^2}
 \left.
 \frac{\partial^2}{\partial u_1\partial u_2}
 F^{(r)}_{g-1,n+1}\big(
 u_1,u_2,z_{[\hat{i}]}
 \big)\right|_{u_1=u_2=z_i}
 \\
 +
 \half \sum_{i=1}^n \frac{z_i^2}{(1-rz_i^r)^2}
 \sum_{\substack{
 g_1+g_2=g\\I\sqcup J=[\hat{i}]}} ^{\rm{stable}}
 \left(
 \frac{\partial}{\partial z_i} 
 F^{(r)}_{g_1,|I|+1}(z_i,z_I)
 \right)
 \left(
  \frac{\partial}{\partial z_i} 
 F^{(r)}_{g_2,|J|+1}(z_i,z_J)
 \right).
 \end{multline}
 Here we use the following convention for
 indices. The index set is 
 $[n]=\{1,2,\dots,n\}$, and for 
 a subset $I\subset [n]$, $z_I = (z_i)_{i\in I}$.
 The hat symbol $\hat{i}$ means the omission of $i$
 from $[n]$. The final summation is over all 
 non-negative integer
 partitions of $g$ and set partitions of $[\hat{i}]$
 subject to the stability conditions
 $2g_1-2+|I|\ge 0$ and $2g_2-2+|J|\ge 0$.
\end{thm}

\begin{rem}
\label{rem:Fgn recursion}
Since $F^{(r)}_{g,n}(z_1,\dots,z_n)\big|_{z_i=0}=0$
for every $i$, the differential recursion 
(\ref{eq:diffrecursion}), which is a 
linear first order partial differential 
equation, uniquely determines
$F^{(r)}_{g,n}$ one by one inductively for all
$(g,n)$ subject to $2g-2+n>0$. 
This generalizes the  result of \cite{MZ}
to the orbifold case.
\end{rem}

The differential recursion of Theorem~\ref{thm:Fgn recursion} is obtained by taking the Laplace 
transform of the cut-and-join equation for $H_{g,n}^{(r)}(\vec{\mu})$. The $r$-Lambert curve itself, \eqref{eq:rLamb}, is obtained by
computing the Laplace transform of $H_{0,1}^{(r)}(\mu)$.

Our third theorem concerns the existence of a quantum curve for orbifold Hurwitz numbers. Since the $r$-Lambert curve has a global 
parameter $z$, the algebraic $K$-theory condition required for the existence of the quantization (see for instance \cite{GS}) is automatically
satisfied, and we have the following result.
\begin{thm}
\label{thm:quantum curve}
The partition function of the orbifold
Hurwitz numbers is given by 
\begin{equation}
\label{eq:partition}
Z^{(r)}(z,\hbar) = 
\exp\left(
\sum_{g=0} ^\infty \sum_{n=1} ^\infty
\frac{1}{n!}\hbar^{2g-2+n}F^{(r)}_{g,n}(z,z,\dots,z)
\right).
\end{equation}
It satisfies the following system of (an infinite-order)
linear differential equations.
\begin{align}
\label{eq:P}
 \left(
 \hbar D - e^{r \left(-w+\frac{r-1}{2}\hbar\right)}
 e^{r\hbar D}
 \right)
 Z^{(r)}(z,\hbar) &=0,
\\
\label{eq:Q}
\left(
\frac{\hbar}{2}
D^2-\left(\frac{1}{r}+\frac{\hbar}{2}
\right) D
-\hbar\frac{\partial}{\partial \hbar}
\right)
Z^{(r)}(z,\hbar) &=0,
\end{align}
where
\begin{equation*}
D=\frac{z}{1-rz^r}\frac{\partial}
{\partial z}=x\frac{\partial}{\partial x}
= -\frac{\partial}{\partial w}.
\end{equation*}
Let the differential operator of (\ref{eq:P})
(resp.~(\ref{eq:Q})
 be denoted by $P$  (resp.~$Q$). Then 
we have the commutator relation
\begin{equation}
\label{eq:[P,Q]}
[P,Q]=P.
\end{equation}
The semi-classical limit of each of the 
equations (\ref{eq:P}) or (\ref{eq:Q})
recovers the $r$-Lambert curve
(\ref{eq:r-Lambert}).
\end{thm}

\begin{rem}
\label{rem:MSS}
The Schr\"odinger equation (\ref{eq:P})
is established in \cite{MSS}.
\end{rem}

\begin{rem}
\label{rem:r=1}
The above theorem
is a generalization of
  \cite[Theorem~1.3]{MS}
  for an arbitrary $r>0$. The restriction
  $r=1$ reduces to the simple Hurwitz case.
\end{rem}

Our final result establishes the prediction from the infinite framing limit that the free energies \eqref{eq:Fgn} should satisfy the integral recursion of Eynard and Orantin with spectral curve the $r$-Lambert curve \eqref{eq:rLamb}. More precisely, it is the symmetric differential forms
\begin{equation}
\label{eq:Wgnr}
W_{g,n}^{(r)}(z_1,\dots,z_n):=
d_1 d_2\cdots d_n F_{g,n}^{(r)}(z_1,\dots,z_n)
\end{equation}
that should satisfy the Eynard-Orantin 
integral recursion on the $r$-Lambert curve.  We establish this fact in the next theorem.\footnote{We refer to \cite{MS} for the precise 
mathematical formulation of the 
Eynard-Orantin recursion formalism.}

\begin{rem}
The significance of the integral formalism
is its universality. The differential equation
(\ref{eq:diffrecursion}) takes a 
 different form
depending on the counting problem, whereas
the integral formula (\ref{eq:r-EO})
depends only on the 
choice of the spectral curve. 
\end{rem}

The Eynard-Orantin integral recursion
requires  a set of geometric data from the 
$r$-Lambert curve, given in parameteric form by $x(z)=z e^{-z^r}$, $y(z)=z^r$. 
The function $x(z)$ has $r$ critical points
at $1-rz^r=0$. Let
$\{p_1,\dots,p_r\}$ be the list of
these critical points. Since $dx=0$ has a simple
zero at each $p_j$,  the map
$x(z)$ is locally a double-sheeted covering
around $z=p_j$. 
We denote by $s_j$ the deck transformation 
on a small neighborhood of $p_j$.

\begin{thm}
\label{thm:r-EO}
For the stable range $2g-2+n>0$, the symmetric
differential forms satisfy the following
integral recursion 
formula.
\begin{multline}
\label{eq:r-EO}
W_{g,n}^{(r)}(z_1,\dots,z_n)
=\frac{1}{2\pi i}\sum_{j=1} ^{r}
\oint_{\gam_j}K_j(z,z_1)
\Bigg[
W_{g-1,n+1}^{(r)}\big(z,s_j(z),z_2,\dots,z_n\big)
\\
+
\sum_{i=2}^n 
\left(
W_{0,2}^{(r)}(z,z_i)
\tensor
W_{g,n-1}^{(r)}
\big(s_j(z),z_{[\hat{1},\hat{i}]}\big)
+
W_{0,2}^{(r)}\big(s_j(z),z_i\big)
\tensor
W_{g,n-1}^{(r)}
\big(z,z_{[\hat{1},\hat{i}]}\big)
\right)
\\
+
\sum_{\substack{g_1+g_2=g\\
I\sqcup J=\{2,\dots,n\}}}
^{\rm{stable}}
W_{g_1,|I|+1}^{(r)}\big(z,z_I\big)
\tensor
W_{g_2,|J|+1}^{(r)}\big(s_j(z),z_J\big)
\Bigg].
\end{multline}
Here the integration is taken with respect to 
$z$ along a small simple closed loop $\gam_j$
around $p_j$. The integration kernel is defined by
\begin{equation}
\label{eq:kernel}
K_j(z,z_1)= \half\;
\frac{1}{W_{0,1}^{(r)}\big(s_j(z_1)\big)
-W_{0,1}^{(r)}(z_1)}
\tensor
\int_{z}^{s_j(z)}W_{0,2}^{(r)}(\;\cdot\;,z_1).
\end{equation}

\end{thm}

\begin{rem}
The proof is based on the idea of \cite{EMS}. 
The notion of the \emph{principal part}
of  meromorphic differentials
 plays a key role in converting the
Laplace transform of the cut-and-join
equation into a residue formula. We
generalize this technique to a more
suitable one that works
for the current orbifold case.
\end{rem}

\begin{rem}
When our manuscript was being finalized, we noticed 
an extremely interesting paper \cite{DLN}. 
The authors of 
\cite{DLN} derive the same
spectral curve using a concrete
graph counting argument, and establish
Theorem~\ref{thm:Fgn recursion} independently. 
They
also claim to have proved our
Theorem~\ref{thm:r-EO}. Although they have
the right strategy, 
their proof  as written is in error. 
\cite[Lemma~13]{DLN} does not 
hold, while it is used
in the key step of proving \cite[Eqn.(22)]{DLN}.
\end{rem}

\subsection{Outline}

The paper is organized as follows.
Section~\ref{sect:Hurwitz} reviews the
orbifold Hurwitz numbers. The key formulas
we use in this paper are the 
ELSV-type formula (\ref{eq:ELSV}) of
\cite{JPT} and the cut-and-join equation
(\ref{eq:CAJ}).
Section~\ref{sect:vertex} is devoted to the infinite framing relation between orbifold Hurwitz numbers and Gromov-Witten 
theory of $[ \bC^3\big/(\bZ/r\bZ)]$.
We then calculate the Laplace transform
of the orbifold Hurwitz numbers and 
prove Theorem~\ref{thm:Fgn recursion}
in Section~\ref{sect:Laplace}.
Section~\ref{sect:properties} lists some
properties enjoyed by the free energies. 
The quantum curve of the 
$r$-Lambert curve is studied in 
Section~\ref{sect:quantum}.
The proof of Theorem~\ref{thm:r-EO} is 
given in Section~\ref{sect:EO}.

\section{The orbifold Hurwitz numbers}
\label{sect:Hurwitz}

The polynomial behavior of simple
Hurwitz numbers $H_{g,n}(\vec{\mu})$
\cite{GJ,V} as a function in $\vec{\mu}$ has
been a long mystery. 
The 
\emph{polynomiality}
has become manifest in the 
Ekedahl-Lando-Shapiro-Vainshtein
 formula \cite{ELSV}
that relates simple Hurwitz numbers and the
Hodge integrals on the Deligne-Mumford moduli
$\Mbar_{g,n}$. 
Another manifestation of the polynomiality is
found in \cite{MZ}, where it is established
that the Laplace transform
\begin{equation}
\label{eq:FgnH}
F_{g,n}(t_1,\dots,t_n)
=\sum_{\vec{\mu}\in\bZ_+^n}
H_{g,n}(\vec{\mu}) e^{-\la \vec{w},\vec{\mu}\ra}
\end{equation}
is a polynomial of degree $3(2g-2+n)$
in  $t_i$-variables.
Here the variables are related by
$$
e^{-w} = ze^{-z},\qquad z=\frac{t-1}{t}.
$$

The
\emph{orbifold Hurwitz numbers}
$H_{g,n}^{(r)}(\vec{\mu})$ no longer exhibits
the same polynomiality.  But it shows   a 
\emph{piecewise} polynomial behavior. 
Indeed, we can define 
$H_{g,n}^{(r)}(\vec{\mu})$ as a double Hurwitz
number, which is the automorphism 
weighted count of the topological types of
double Hurwitz covers 
$\varphi:C\lrar \bP^1$. Here $C$ is a connected
nonsingular curve of genus $g$, and $\varphi$ is 
a holomorphic map that has $n$ labeled 
poles of orders $\vec{\mu}$,  $m$
\emph{unlabeled} zeros of degree $r$, and
all other critical points are unlabeled simple 
ramification points. 
The number of zeros is given by
\begin{equation}
\label{eq:m}
m = \frac{\mu_1+\cdots+\mu_n}{r}.
\end{equation}
This is a special case of the fully general
double Hurwitz numbers
$H_{g,m,n}(\vec{\mu},\vec{\nu})$ of 
arbitrary zeros and poles and otherwise simply
ramified. We refer to \cite{CJM,GJV} for 
further discussions on the piecewise polynomiality.

Reflecting the chamber structure of the 
polynomiality, the ELSV-type formula
for orbifold Hurwitz numbers is
 more complicated than the original case.
 The following formula is established in
 Johnson-Pandharipande-Tseng 
 \cite{JPT}.

\begin{thm}[\cite{JPT}]
\label{thm:JPT}
The orbifold Hurwitz number has an expression 
in terms of linear Hodge integrals as follows:
\begin{equation}
\label{eq:ELSV}
H_{g,n}^{(r)}(\mu_1,\dots,\mu_n)
=
r^{1-g +\sum_{i=1}^n \la \frac{\mu_i}{r}\ra}
\int_{\Mbar_{g,-\vec{\mu}}(BG)}
\frac{\sum_{j\ge 0}(-r)^j \lam_j}
{\prod_{i=1}^n(1-\mu_i\psi_i)}
\prod_{i=1}^n
 \frac{\mu_i^{\lfloor\frac{\mu_i}{r}\rfloor}}
{\lfloor\frac{\mu_i}{r}\rfloor!}.
\end{equation}
Here, $G=\bZ/r\bZ$, and $BG$ is the classifying
space of $G$. 
The floor and the fractional part of $q\in\bQ$
is given by
$q=\lfloor q \rfloor +\la q\ra$.
$\Mbar_{g,-\vec{\mu}}(BG)$ denotes
the moduli space of stable morphisms to $BG$
from a stable curve of genus $g$ and 
$n$ smooth points on it, with a prescribed
monodromy data $-\vec{\mu}$.  The 
vector $-\vec{\mu}$, as the monodromy data,   is
identified with the residue class
$$
-\vec{\mu} \mod r =
(-\mu_1\mod r,\dots,-\mu_n\mod r)
\in G^r
$$
 at each marked point. We fix a character
$$
G=\bZ/r\bZ \owns k \longmapsto
 e^{\frac{2\pi i k}{r}}
\in \bC^*.
$$
This defines a line bundle on 
$[C,(p_1,\dots,p_n)]\in \Mbar_{g,n}$, 
and the choice of the monodromy 
data $\vec{\mu}\in G^r$ determines a covering
$\tilde{C}\lrar C$ as a multi-section 
of this line bundle. All these data give a
point of the moduli stack 
$\Mbar_{g,-\vec{\mu}}(BG)$, and the Hodge `bundle'
$\bE$ is defined on it by assigning the fiber
$H^0\big(\tilde{C}, 
K_{\tilde{C}}
\big)$ to this point, where $K_{\tilde{C}}$ is the 
canonical sheaf. We then define
$$
\lam_j = c_j(\bE)\in H^{2j}\big(
\Mbar_{g,-\vec{\mu}}(BG),\bQ\big).
$$
The $\psi$-classes on $\Mbar_{g,-\vec{\mu}}(BG)$ are
the pull-back of the standard  tautological
cotangent classes
on $\Mbar_{g,n}$ via the natural 
forgetful morphism
$$
\Mbar_{g,-\vec{\mu}}(BG)\lrar \Mbar_{g,n}.
$$
\end{thm}

The cut-and-join equation of orbifold
Hurwitz numbers
$H_{g,n}^{(r)}(\mu_1,\dots,\mu_n)$
is derived from the analysis of the geometric 
deformation of confluence of one of the simple 
ramification points with  $\infty\in \bP^1[r]$.
In terms of the monodromy data, the deformation
corresponds to multiplying a transposition to 
the product of $n$ disjoint cycles of 
type $(\mu_1,\dots,\mu_n)$ that determine
the ramification profile above $\infty$.
Therefore, the geometric situation in our
orbifold context does not
change from the usual simple Hurwitz number
case. As a result, the exact same proof of the 
original case (see for example, \cite{MZ} and
\cite{Zhu}) 
applies to establish the following. 

\begin{thm}[Cut-and-join equation]
The orbifold
Hurwitz numbers
$H_{g,n}^{(r)}(\mu_1,\dots,\mu_n)$
satisfy the following equation.
\begin{multline}
\label{eq:CAJ}
s H_{g,n}^{(r)}(\mu_1,\dots,\mu_n)
= \half\sum_{i\ne j} (\mu_i+\mu_j)
H_{g,n-1}^{(r)}
\left(
\mu_i+\mu_j,\mu_{[\hat{i},\hat{j}]}
\right)
\\
+\half\sum_{i=1}^n
\sum_{\a+\b=\mu_i}
\a\b 
\left[
H_{g-1,n+1}^{(r)}
\left(
\a,\b,\mu_{[\hat{i}]}
\right)
+\sum_{\substack{g_1+g_2=g\\
I\sqcup J=[\hat{i}]}}
H_{g_1,|I|+1}^{(r)}\big(\a,\mu_I\big)
H_{g_2,|J|+1}^{(r)}\big(\b,\mu_J\big)
\right].
\end{multline}
Here 
\begin{equation}
\label{eq:s}
s=s(g,\vec{\mu}) = 2g-2+n +
 \frac{\mu_1+\cdots+\mu_n}{r}
\end{equation}
 is the number of 
simple ramification point given by the
Riemann-Hurwitz formula, 
and we use the convention that for any
subset $I\subset [n] = \{1,2,\dots,n\}$, 
$\mu_I=(\mu_i)_{i\in I}$. The hat notation
$\hat{i}$ indicates that the index $i$ is removed. 
The last summation is over all 
partitions of $g$ and set partitions of
$[\hat{i}] = \{1, \dots, i-1,i+1,\dots,n\}$. 
\end{thm}

\section{The infinite framing limit of the orbifold topological vertex}
\label{sect:vertex}

The realization that generating functions for simple Hurwitz numbers satisfy the Eynard-Orantin recursion for the Lambert curve \eqref{eq:Lambert} originated from topological string theory. More precisely, the argument put forward in \cite{BM} was that generating functions for simple Hurwitz numbers can be obtained in the infinite framing limit of the topological vertex generating functions in open Gromov-Witten theory. The remodeling conjecture of \cite{BKMP} then asserts that the topological vertex generating functions should satisfy the Eynard-Orantin recursion for the framed curve mirror to $\mathbb{C}^3$, whose infinite framing limit is precisely the Lambert curve. Hence, it follows from the remodeling conjecture that generating functions for simple Hurwitz numbers should also satisfy the Eynard-Orantin recursion for the limiting curve, that is, the Lambert curve \eqref{eq:Lambert}. In the context of simple Hurwitz numbers, this conjecture has been proved in \cite{BEMS,EMS}, and the remodeling conjecture for the topological vertex has also been proved in \cite{Zhou3} following similar methods.

In this section, we argue that there exists a similar story for orbifold Hurwitz numbers. We show that generating functions for orbifold Hurwitz numbers can be obtained in the infinite framing limit of the orbifold topological vertex generating functions in open orbifold Gromov-Witten theory. By the remodeling conjecture, the latter are expected to satisfy the Eynard-Orantin recursion for the curve mirror to the orbifolds. We show that the infinite framing limit of these curves reproduce the $r$-Lambert curve \eqref{eq:rLamb}, therefore suggesting that generating functions for orbifold Hurwitz numbers should satisfy the Eynard-Orantin recursion for the $r$-Lambert curve. We will prove this result in section \ref{sect:EO}.

\subsection{Open orbifold Gromov-Witten theory}

\subsubsection{The geometry}

We consider the toric Calabi-Yau orbifold $ X = [\mathbb{C}^3 / (\mathbb{Z} / r \mathbb{Z})]$, where $\mathbb{Z} / r \mathbb{Z}$ acts on the three complex coordinates of $\mathbb{C}^3$ as:
\begin{equation}\label{eq:action}
(z_1, z_2, z_3) \mapsto (\alpha z_1, \alpha^s z_2, \alpha^{-s-1} z_3), \qquad \alpha = \exp\left( \frac{2 \pi i}{r} \right), \qquad s \in \mathbb{Z}.
\end{equation}
The rays for the fan of $X$ can be taken to be:
\begin{equation}\label{eq:rays}
\begin{pmatrix}0\\0\\1 \end{pmatrix}, \qquad\begin{pmatrix}0\\1\\1 \end{pmatrix},  \qquad \begin{pmatrix}r\\-s\\1 \end{pmatrix}.
\end{equation}
The \emph{fan triangulation} of $X$ is the intersection of its fan with the plane at $z=1$, which is shown in red in figure \ref{fig:fan}. Its dual diagram is the \emph{toric diagram} (or \emph{web diagram}) of $X$, which is shown in blue. For a good pedagogical introduction to web diagrams and fan triangulations of toric Calabi-Yau orbifolds, see for instance Appendix B in \cite{BCY}.

\begin{figure}[htb]
\centerline{\epsfig{file=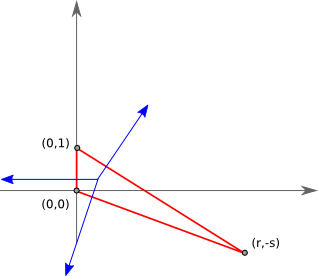,width=3in}}
\caption{The fan triangulation (in red) and toric diagram (in blue) for $X =[\mathbb{C}^3 / (\mathbb{Z} / r \mathbb{Z})]$, with the action of  $\mathbb{Z} / r \mathbb{Z}$ specified by \eqref{eq:action}. The fan triangulation is the Newton polygon for the curve mirror to $X$.}
    \label{fig:fan}
\end{figure}

\subsubsection{Open orbifold Gromov-Witten invariants}

We are interested in open orbifold Gromov-Witten theory with target space $X$. Open Gromov-Witten invariants provide a virtual count of stable maps from Riemann surfaces with boundaries to a target space $X$. In addition to $X$, one must specify a Lagrangian submanifold $L \subset X$ where the boundary of the domain curve is required to lie. We choose our Lagrangian submanifold to be as constructed originally in \cite{AKMV, AV1}, intersecting the $z_1$ coordinate axis of $X$. In the language of \cite{Ross}, we are studying the ``effective one-leg $\mathbb{Z} / r \mathbb{Z}$ orbifold topological vertex'': one-leg because we consider only one Lagrangian submanifold for the boundary condition, and effective because our Lagrangian submanifold intersects the $z_1$ leg of $X$, where the action of $\mathbb{Z} / r \mathbb{Z}$ is effective. This is known as the orbifold topological vertex, because this type of geometry provides a building block that can be used to construct open/closed Gromov-Witten theory for any toric Calabi-Yau orbifolds, just as the original topological vertex of \cite{AKMV} is the building block to construct open/closed Gromov-Witten theory of toric Calabi-Yau manifolds.

We will not give a precise definition of open Gromov-Witten theory here; we refer the interested reader to \cite{BC,KL,LLLZ,Ross}. Roughly speaking, in \cite{KL}, Katz and Liu were the first to construct a tangent/obstruction theory for the moduli space of open stable maps to toric Calabi-Yau manifolds. The construction was generalized to orbifolds in \cite{BC}, and then in full generality by Ross in \cite{Ross}. An important point is that the moduli theory is only defined via localization with respect to a torus action on the moduli space, induced from a torus action on the target space $X$. There is a choice of weights involved in the choice of torus action on the target space $X$, and it turns out that open Gromov-Witten invariants do depend on this choice of weight. More precisely, they depend on a residual integer $f \in \mathbb{Z}$, which is known as the \emph{framing} of the open Gromov-Witten invariants (in fact, in the context of orbifolds, $f \in \frac{1}{r} \mathbb{Z}$). To make contact with the notation of \cite{BC,Ross}, here we choose the weights for the torus action with respect to which we localize to be
\begin{equation}\label{eq:weights}
\left(\frac{\hbar}{r},  f \hbar, - f \hbar - \frac{\hbar}{r} \right),
\end{equation}
just as in \cite{BC}.\footnote{Note that our $f$ has a minus sign difference with \cite{BC}, which is consistent with the framing that we will introduce for the mirror curve later on.}
In the non-equivariant limit in which we will evaluate Gromov-Witten invariants, we set $\hbar = 1$.

After localization, open Gromov-Witten theory becomes a theory of stable maps $\varphi: \Sigma \to X$,
 where $\Sigma$ is a compact genus $g$ Riemann surface with $n$ disks attached at $n$ (possibly $k$-twisted) distinct nodes. The map $\varphi$ contracts the compact component to the origin of the target orbifold $X$, while the $n$ boundaries of the disks are mapped to the Lagrangian submanifold $L$. Each disk is mapped with a given winding number $\mu_i \in \mathbb{Z}$, $i=1,\ldots, n$. Thus, the data encoding a map $\varphi$ is the genus $g$ of the  domain curve, a partition $\vec{\mu}$ of length $\ell(\mu) = n$ specifying the winding numbers of the disks, and a vector $\vec{k}$ of integers $0 < k_i \leq r$ specifying the twisting of the attachment points.

In fact, as shown in \cite{BC}, for the theory to be $k$-twisted equivariant, the twisting vector $\vec{k}$ is not independent from the winding numbers $\vec{\mu}$: we must require that
\begin{equation}
\mu_i \equiv k_i \text{ mod $r$},
\end{equation}
which fully specifies $\vec{k}$ in terms of $\vec{\mu}$.

\begin{rem}
We remark here that we do not allow insertions, that is, stacky points on the compact components of the domain curves, aside from the attachment points of the disks. It would be interesting to study the Eynard-Orantin recursion for the orbifold topological vertex with insertions, and its infinite framing limit. We hope to report on that in the near future.
\end{rem}

\subsubsection{The orbifold topological vertex}

Under the assumptions described above, we can construct the effective one-leg orbifold topological vertex $V^{(r,s)}_{g,n}(\vec{\mu};f)$, which computes the open orbifold Gromov-Witten invariants of $X$ from genus $g$ domain curves with $n$ disks with winding numbers specified by the partition $\vec{\mu}$. We form orbifold topological vertex generating functions:
\begin{equation}\label{eq:genGW}
G^{(r,s)}_{g,n}[x_1,\ldots,x_n;f] = \sum_{\vec{\mu} \in \mathbb{Z}^n_+} V^{(r,s)}_{g,n}(\vec{\mu};f) \prod_{i=1}^n x_i^{\mu_i}. 
\end{equation}

One of the main results of \cite{BC, Ross} is that the orbifold topological vertex $V^{(r,s)}_{g,n}(\vec{\mu};f)$ has an explicit formula in terms of Hurwitz-Hodge integrals over the moduli space $\Mbar_{g,-\vec{\mu}}(B G)$, with $G = \mathbb{Z} / r \mathbb{Z}$. More precisely, in the non-equivariant limit, from the work of \cite{BC,Ross,Zhong} we extract the following formula for the orbifold topological vertex described above:\footnote{Note that the overall sign in $V^{(r,s)}_{g,n}(\vec{\mu};f)$ differs from \cite{BC,Ross}; as mentioned in \cite{BC}, there are ambiguities with minus signs in open Gromov-Witten theory. Here, we fixed the overall minus sign such that it is consistent with the infinite framing limit that we will study. It would be interesting to investigate this issue of minus signs further.}
\begin{align}
V^{(r,s)}_{g,n}(\vec{\mu};f) =& (-1)^{g-1 + \sum_{i=1}^n \langle - \frac{\mu_i(s+1)}{r} \rangle} r^{n -  \sum_{i=1}^n \delta_{\langle\frac{\mu_i}{r}\rangle,0}} (f)^{\sum_{i=1}^n\delta_{\langle \frac{\mu_i s}{r}\rangle,0}} \left( -f - \frac{1}{r} \right)^{\sum_{i=1}^n \delta_{\langle - \frac{(s+1)\mu_i}{r} \rangle,0}} \nonumber\\
&\times \prod_{i=1}^n \left( \frac{1}{\mu_i^{\lfloor \frac{\mu_i}{r} + \langle \frac{\mu_i s}{r} \rangle - \frac{1}{t_{\rm eff}} \rfloor - \lfloor \frac{\mu_i}{r} \rfloor} \lfloor \frac{\mu_i}{r} \rfloor !} \prod_{j=1}^{\lfloor \frac{\mu_i}{r} + \langle \frac{\mu_i s}{r} \rangle - \frac{1}{t_{\rm eff}}\rfloor} \left(  f \mu_i - \langle \frac{\mu_i s}{r} \rangle + j \right) \right) \nonumber\\
& \times \int_{\Mbar_{g,-\vec{\mu}}(B G)} \frac{\Lambda_g^{\vee, \alpha}\left(\frac{1}{r}\right)\Lambda_g^{\vee, \alpha^s}\left(f\right)\Lambda_g^{\vee, \alpha^{-s-1}}\left(-f-\frac{1}{r}\right)}{\prod_{i=1}^n ( 1- \mu_i \psi_i)},
\label{eq:orbTV}
\end{align}
where we used the Kronecker delta symbol notation:
\begin{equation}
\delta_{t,0} = \begin{cases} 0 \text{ if } t \neq 0, \\
1 \text{ if } t = 0,\end{cases}
\end{equation}
and we defined a rational number
\begin{equation}
t_{\rm eff} := \frac{r}{\text{gcd}(\mu_i,r)}.
\end{equation}
We also used the notation:
\begin{equation}
\Lambda_g^{\vee,\alpha^k} (u) = u^{\text{rk}( \mathbb{E}_{\alpha^k})} \sum_{i=0}^{\text{rk} (\mathbb{E}_{\alpha^k})} \left( -\frac{1}{u} \right)^i
 \lambda_{i,\alpha^k},
\end{equation}
where $\mathbb{E}_{\alpha^k}$ is the Hodge bundle corresponding to the representation of $\mathbb{Z} / r \mathbb{Z}$ given by
\begin{equation}
\phi_{\alpha^k}: \mathbb{Z} / r \mathbb{Z} \to \mathbb{C}^*, \qquad \phi_{\alpha^k}(1) ={\alpha^k} = \exp \left( \frac{2 \pi i k }{r} \right),
\end{equation}
and 
\begin{equation}
 \lambda_{i,\alpha^k} = c_i \left(\mathbb{E}_{\alpha^k} \right)
\end{equation}
are its Chern classes.

\subsection{The infinite framing limit of the generating functions}

With this explicit formula for the orbifold topological vertex, we can study the limit of the generating functions \eqref{eq:genGW} when the framing $f$ goes to infinity. What we show is that the infinite framing limit of the orbifold topological vertex reproduces precisely the orbifold Hurwitz numbers defined previously.

\begin{thm}\label{thm:limit}
Consider the generating functions $G^{(r,s)}_{g,n}[x_1,\ldots,x_n;f]$ defined in \eqref{eq:genGW}, with the orbifold topological vertex $V^{(r,s)}_{g,n}(\vec{\mu};f)$ given by \eqref{eq:orbTV}. Then:
\begin{equation}
\lim_{f \to \infty} \left((-1)^n f^{2-2g-n} G^{(r,s)}_{g,n}\left[\frac{x_1}{f^{1/r}},\ldots,\frac{x_n}{f^{1/r}};f \right] \right) = F_{g,n}^{(r)}[x_1, \ldots, x_n],
\end{equation}
where $F_{g,n}^{(r)}[x_1, \ldots, x_n]$ is the generating functions for orbifold Hurwitz numbers defined in \eqref{eq:Fgn in x}, with the orbifold Hurwitz numbers $H_{g,n}^{(r)}(\vec{\mu})$ satisfying \eqref{eq:ELSV}. 
\end{thm}

\begin{proof}
Let us consider the leading order term in a large $f$ expansion of the orbifold topological vertex $V^{(r,s)}_{g,n}(\vec{\mu};f)$ in \eqref{eq:orbTV}. Let us consider the Hurwitz-Hodge integral first. In the large $f$ limit, it is easy to see that
\begin{equation}
\Lambda_g^{\vee, \alpha^s}\left(f\right)\Lambda_g^{\vee, \alpha^{-s-1}}\left(-f-\frac{1}{r}\right) \simeq (f)^{\text{rk}(\mathbb{E}_{\alpha^{s}})}(- f)^{\text{rk}(\mathbb{E}_{\alpha^{-s-1}})},
\end{equation}
since all other terms will be suppressed by powers of $1/f$. We can compute the rank of the Hodge bundles over $\Mbar_{g,-\vec{\mu}}(B G)$ using orbifold Riemann-Roch. We get that\footnote{To be precise, we should consider separately the cases when the moduli space has a component with trivial monodromy (see for instance \cite{JPT}). But since the same formulae are valid in the end, for the sake of clarity we will not treat these cases separately.}
\begin{align}
\text{rk}(\mathbb{E}_{\alpha^{s}} )=& g - 1 + \sum_{i=1}^n \langle -\frac{ \mu_i s}{r} \rangle, \\
\text{rk}(\mathbb{E}_{\alpha^{-s-1}}) =& g - 1 + \sum_{i=1}^n \langle \frac{ \mu_i (s+1)}{r} \rangle.
\end{align}
Moreover, we can write
\begin{equation}
\Lambda_g^{\vee, \alpha}\left(\frac{1}{r}\right) = r^{-\text{rk}(\mathbb{E}_{\alpha} )} \sum_{i=0}^{\text{rk} (\mathbb{E}_{\alpha})} \left( - r \right)^i  \lambda_{i,\alpha},
\end{equation}
and we compute
\begin{equation}
\text{rk}(\mathbb{E}_{\alpha} )  = g - 1 + \sum_{i=1}^n \langle -\frac{ \mu_i}{r} \rangle.
\end{equation}
Thus, the Hurwitz-Hodge integral in the third line of \eqref{eq:orbTV} has the following leading order term in a large $f$ expansion:
\begin{equation}
r^{1-g- \sum_{i=1}^n \langle -\frac{ \mu_i}{r} \rangle} (-1)^{g-1 + \sum_{i=1}^n \langle \frac{ \mu_i (s+1)}{r} \rangle} f^{2g-2 + \sum_{i=1}^n \left(\langle -\frac{ \mu_i s}{r} \rangle+  \langle \frac{ \mu_i (s+1)}{r} \rangle\right)}
\int_{\Mbar_{g,-\vec{\mu}}(B G)} \frac{\sum_{j \geq 0} (-r)^j \lambda_{j,\alpha} }{\prod_{i=1}^n ( 1- \mu_i \psi_i)}.
\end{equation}

The first line of \eqref{eq:orbTV} has leading order term given by
\begin{equation}
 (-1)^{g-1 + \sum_{i=1}^n \langle - \frac{\mu_i(s+1)}{r} \rangle}r^{n -  \sum_{i=1}^n \delta_{\langle\frac{\mu_i}{r}\rangle,0}} f^{\sum_{i=1}^n\delta_{\langle \frac{\mu_i s}{r}\rangle,0}} (-f)^{\sum_{i=1}^n \delta_{\langle - \frac{(s+1)\mu_i}{r} \rangle,0}}.
\end{equation}
As for the second line in \eqref{eq:orbTV}, the leading order term is
\begin{equation}
\prod_{i=1}^n \left( \frac{\mu_i^{\lfloor \frac{\mu_i}{r} \rfloor}}{\lfloor \frac{\mu_i}{r} \rfloor !}f^{ \lfloor \frac{\mu_i}{r} + \langle \frac{\mu_i s}{r} \rangle - \frac{1}{t_{\rm eff}}\rfloor}\right)
\end{equation}

To get our final answer we must combine these three lines together. For the exponent of the overall factor of $f$, we notice that
\begin{align}
&2g-2\sum_{i=1}^n \left(\langle -\frac{ \mu_i s}{r} \rangle+  \langle \frac{ \mu_i (s+1)}{r} \rangle +\delta_{\langle \frac{\mu_i s}{r}\rangle,0}+\delta_{\langle - \frac{(s+1)\mu_i}{r} \rangle,0}+ \lfloor \frac{\mu_i}{r} + \langle \frac{\mu_i s}{r} \rangle - \frac{1}{t_{\rm eff}}\rfloor \right) \nonumber\\
=& 2g-2+n + \sum_{i=1}^n \left(1 -  \langle \frac{ \mu_i s}{r} \rangle-  \langle \frac{ -\mu_i (s+1)}{r} \rangle+ \lfloor \frac{\mu_i}{r} + \langle \frac{\mu_i s}{r} \rangle - \frac{1}{t_{\rm eff}}\rfloor \right) \nonumber\\
=& 2g-2+n + \sum_{i=1}^n \left(1 + \frac{\mu_i}{r} + \lfloor - \frac{\mu_i (s+1)}{r} \rfloor + \lfloor \frac{\mu_i(s+1)}{r} - \frac{1}{t_{\rm eff}} \rfloor \right) \nonumber\\
=& 2g-2+n + \sum_{i=1}^n \frac{\mu_i}{r}.
\end{align}
The last equality follows because:
\begin{equation}
 \lfloor - \frac{\mu_i (s+1)}{r} \rfloor + \lfloor \frac{\mu_i(s+1)}{r} - \frac{1}{t_{\rm eff}} \rfloor =
\begin{cases}
 \lfloor  - \frac{1}{t_{\rm eff}} \rfloor =\lfloor  - \frac{\text{gcd}(\mu_i,r)}{r} \rfloor =  -1 & \text{for $\frac{\mu_i(s+1)}{r} \in \mathbb{Z}$},\\
-1 + \lfloor \langle \frac{\alpha_i (s+1)}{t_{\rm eff}} \rangle - \frac{1}{t_{\rm eff}}\rfloor  = -1 &\text{for $\frac{\mu_i(s+1)}{r} \notin \mathbb{Z}$,}
\end{cases}
\end{equation}
where  $\alpha_i := \frac{\mu_i t_{\rm eff}}{r} = \frac{\mu_i}{\text{gcd}(\mu_i,r)}\in \mathbb{Z}$.

As for the exponent of the factor in $r$, we get
\begin{align}
1-g+n-\sum_{i=1}^n \left( \langle -\frac{ \mu_i}{r} \rangle+ \delta_{\langle\frac{\mu_i}{r}\rangle,0} \right)
=& 1-g+n - \sum_{i=1}^n \left( 1 - \langle \frac{ \mu_i}{r} \rangle \right) \nonumber\\
=&1-g+ \sum_{i=1}^n \langle \frac{ \mu_i}{r} \rangle.
\end{align}
Finally, the overall minus sign has exponent:
\begin{equation}
2g-2+ \sum_{i=1}^n \left(\langle- \frac{\mu_i(s+1)}{r} \rangle + \langle \frac{\mu_i(s+1)}{r} \rangle+\delta_{\langle - \frac{(s+1)\mu_i}{r} \rangle,0}\right) = 2g-2+n.
\end{equation}
Putting these together, we obtain that the leading term of \eqref{eq:orbTV} as $f$ is large is
\begin{equation}
(-1)^n f^{2g - 2 + n + \sum_{i=1}^n \frac{\mu_i}{r}} H_{g,n}^{(r)}(\vec{\mu}),
\end{equation}
where the orbifold Hurwitz numbers $H_{g,n}^{(r)}(\vec{\mu})$ are defined in \eqref{eq:ELSV}. Therefore, it follows that the generating functions satisfy
\begin{equation}
\lim_{f \to \infty} \left((-1)^n f^{2-2g-n} G^{(r,s)}_{g,n}\left[\frac{x_1}{f^{1/r}},\ldots,\frac{x_n}{f^{1/r}};f \right] \right) = F_{g,n}^{(r)}[x_1, \ldots, x_n].
\end{equation}
\end{proof}

\subsection{The remodeling conjecture and the Eynard-Orantin recursion}

An interesting implication of the infinite framing limit studied in the previous subsection is the existence of a recursive structure for orbifold Hurwitz numbers, which follows from the remodeling conjecture of \cite{BKMP}.

The remodeling conjecture asserts that the differentials $d_1 d_2 \cdots d_n G^{(r,s)}_{g,n}\left[x_1, \ldots, x_n;f \right] $ can be resummed as symmetric differential forms living on the complex curve mirror to the orbifold $X$, and that they satisfy the Eynard-Orantin recursion (which was defined in the introduction) for this particular spectral curve. Through the infinite framing limit of the generating functions, this conjecture implies that the generating functions for orbifold Hurwitz numbers should also satisfy the Eynard-Orantin recursion, with spectral curve given by the $r$-Lambert curve.

Recall that the mirror curve to the orbifold $X$ can be read off directly from the fan triangulation of $X$. Indeed, the fan triangulation is the Newton polygon of the mirror curve. In the case of the orbifold $X = [ \mathbb{C}^3/ (\mathbb{Z} / r\mathbb{Z})]$ that we studied in this section, the fan triangulation of $X$, shown in figure \ref{fig:fan}, has vertices $(0,0), (0,1)$ and $(r,-s)$, corresponding to the monomials $1$, $y$ and $x^r y^{-s}$. Therefore, the mirror curve can be written as:
\begin{equation}
C: \{ 1 - y - x^r y^{-s} = 0 \} \subset (\mathbb{C}^*)^2.
\end{equation}
Note that in writing the mirror curve, we have chosen a particular parameterization (in the language of toric geometry, we chose a particular set of rays, \eqref{eq:rays}, for the fan of $X$). We claim that this particular choice of parameterization should correspond to a Lagrangian submanifold intersecting the $z_1$ leg of the orbifold $X$ (we refer the reader to \cite{BKMP,BKMP2,BC,BT} for more on this).

To introduce framing for the mirror curve, we must reparameterize the curve by $(x,y) \mapsto (x y^{-f},y)$ \cite{BKMP} . We then get the framed mirror curve:\footnote{Notice that the framing transformation of \cite{BC} replaces $f$ by $-f$, which is why our choice of torus weights in \eqref{eq:weights} had a minus sign difference with \cite{BC}.}
\begin{equation}\label{eq:fmc}
C_f: \{ y^{s + r f}(1 - y) - x^r = 0 \} \subset (\mathbb{C}^*)^2.
\end{equation}
Note that we denoted the framing by $f$, the same letter as in the previous subsection, but the two may not be precisely equal; they may be related via the addition of a constant (see for instance \cite{BC}). But this will not be important for us, since we are interested in the $f \to \infty$ limit. 

The statement of the remodeling conjecture is that the differentials $d_1 d_2 \cdots d_n G^{(r,s)}_{g,n}\left[x_1, \ldots, x_n;f \right] $ are symmetric differential forms on $C_f$ that satisfy the Eynard-Orantin recursion for the spectral curve $C_f$, with fundamental one-form
\begin{equation}\label{eq:oneform}
d G^{(r,s)}_{0,1}[x;f] = \log y \frac{d x}{x},
\end{equation}
where $x$ and $y$ are related by \eqref{eq:fmc}.

Now what happens in the infinite framing limit? First, we notice that if we define new variables
\begin{equation}
x = \frac{\tilde{x}}{f^{1/r}}, \qquad y = 1 - \frac{\tilde{y}}{f},
\end{equation}
the equation for the framed mirror curve $C_f$ in \eqref{eq:fmc} becomes
\begin{equation}
\tilde{x}^r = \tilde{y} \left( 1 - \frac{\tilde{y}}{f} \right)^{rf} \left( 1 - \frac{\tilde{y}}{f} \right)^s.
\end{equation}
Taking the limit $f \to \infty$, we obtain the curve
\begin{equation}\label{eq:rL2}
\tilde{x}^r = \tilde{y} \mathrm{e}^{- r \tilde{y}},
\end{equation}
which is precisely the equation of the $r$-Lambert curve \eqref{eq:rLamb}!

What does it mean for the recursion satisfied by the generating functions? The one-form that is fundamental for the recursion, \eqref{eq:oneform}, can be rewritten in terms of the new variables $\tilde{x}$ and $\tilde{y}$. It becomes
\begin{equation}
d G^{(r,s)}_{0,1}[x;f] = \log y \frac{d x}{x} = \log \left( 1 - \frac{\tilde{y}}{f} \right) \frac{d \tilde{x}}{\tilde{x}},
\end{equation}
with $x$ and $y$ related by \eqref{eq:fmc}.
If we send $f \to \infty$, the leading order term is
\begin{equation}\label{eq:change}
- \frac{\tilde{y}}{f} \frac{d \tilde{x}}{\tilde{x}},
\end{equation}
with $\tilde{x}$ and $\tilde{y}$ now related through \eqref{eq:rL2}. 

Looking at the explicit form of the recursion, the result of this analysis is that if we consider the infinite framing limit of the differentials
\begin{equation}
d_1 d_2 \cdots d_nF^{(r)}_{g,n}[\tilde{x}_1, \ldots, \tilde{x}_n]:=\lim_{f \to \infty} \left((-1)^n f^{2-2g-n}d_1 d_2 \cdots d_n G^{(r,s)}_{g,n}\left[\frac{\tilde{x}_1}{f^{1/r}}, \ldots, \frac{\tilde{x}_n}{f^{1/r}};f \right]  \right),
\end{equation}
then they should satisfy the Eynard-Orantin recursion with fundamental one-form
\begin{equation}\label{eq:newic}
d F^{(r)}_{0,1}[\tilde{x}] = \tilde{y} \frac{d \tilde{x}}{\tilde{x}},
\end{equation}
where $\tilde{x}$ and $\tilde{y}$ are related by \eqref{eq:rL2}.
The $-1/f$ factor between \eqref{eq:newic} and \eqref{eq:change} is precisely responsible for the $(-1)^n f^{2-2g-n}$ factor in front of the differentials constructed from the recursion.

But we know what these new objects $d_1 d_2 \cdots d_nF^{(r)}_{g,n}[\tilde{x}_1, \ldots, \tilde{x}_n]$ are: in the previous section, we showed that they are precisely the differentials of the generating functions of orbifold Hurwitz numbers! Therefore, if we believe the remodeling conjecture for orbifolds, then we are led to claim that the generating functions for orbifold Hurwitz numbers should also satisfy the Eynard-Orantin recursion, with spectral curve the $r$-Lambert curve \eqref{eq:rLamb}, which can be written in parameteric form as
\begin{equation}
\label{eq:parametric}
\tilde{x} = z \mathrm{e}^{-z^r}, \qquad \tilde{y}=z^r,
\end{equation}
with fundamental one-form
\begin{equation}
d F^{(r)}_{0,1}[\tilde{x}] = \tilde{y} \frac{d \tilde{x}}{\tilde{x}} = z^{r-1} (1-r z^r) d z.
\end{equation}
We will prove this statement in section \ref{sect:EO}.

\section{The Laplace transform of the
orbifold Hurwitz numbers}
\label{sect:Laplace}

In this section we prove 
Theorem~\ref{thm:Fgn recursion}. 
From the remodeling conjecture
 point of view presented in the
previous section,
we see that 
the mirror theory to orbifold Hurwitz numbers
should be  built on the $r$-Lambert curve
(\ref{eq:r-Lambert}). To launch 
the Eynard-Orantin topological recursion
\cite{EO1, MS} for the $r$-Lambert curve as its
spectral curve, we
need to find  the Lagrangian immersion
\begin{equation*}
\begin{CD}
\iota :\Sigma @>>> T^*\bC^*
\\
&&@VV{\pi}V
\\
&&\bC^*
\end{CD}
\end{equation*}
of the open Riemann surface $\Sigma= \bC^*$
given by 
\begin{equation}
\label{eq:spectral}
\begin{cases}
x = z e^{-z^r}\\
y = f(z)
\end{cases}
\qquad z\in \Sigma,
\end{equation}
where $y=f(z)$ is yet to be determined.
We refer to \cite{MS} for a mathematical 
definition of the Eynard-Orantin 
topological recursion theory. 
The recipe of \cite{MS} tells us that the 
Laplace transform of the \emph{disk amplitude}
$H_{0,1}^{(r)}(\mu)$ should determine the 
Lagrangian immersion
by the formula
\begin{equation}
\label{eq:W01=dF01}
W_{0,1}^{(r)}(z) \overset{\text{def}}{=}
 \iota^*(y d\log x) = dF_{0,1}
^{(r)}(z),
\end{equation}
where
$\eta = y d\log x$ on $T^*\bC^*$ is 
the tautological holomorphic $1$-form
on the cotangent bundle $T^*\bC^*$.
In this section we first identify the 
Lagrangian immersion
(\ref{eq:spectral}) from the
computation of $F_{0,1}^{(r)}$.
We learn from \cite{JPT}  that
$$
H_{0,1}^{(r)}(\mu) = 
\frac{\mu^{\lfloor\frac{\mu}{r}\rfloor-2}}
{\lfloor\frac{\mu}{r}\rfloor !}
\quad \text{if} \quad \mu\equiv 0 \mod r,
$$
and $H_{0,1}^{(r)}(\mu) =0$ otherwise. 
Therefore, the $(g,n) = (0,1)$ free energy (\ref{eq:Fgn})
is given by 
$$
F_{0,1}^{(r)} = \sum_{m=1}^\infty 
\frac{(rm)^{m-2}}{m!} x^{rm}.
$$
We note that $F_{0,1}^{(r)}=0$ when $x=0$.

As the ELSV-type formula (\ref{eq:ELSV})
 indicates, the free energy computation
 requires that we need to find
similar infinite sums. We thus introduce the following
auxiliary functions:
\begin{equation}
\begin{aligned}
\label{eq:xi}
\xi_{\ell}^{r,k}(x) 
&= 
\sum_{m=0}^\infty 
\frac{(rm+k)^{m+\ell}}{m!}\; x^{rm+k}, 
\qquad k = 1, 2, \dots, r-1,
\\
\xi_{\ell}^{r,0}(x) 
&= 
\sum_{m=1}^\infty 
\frac{(rm)^{m+\ell}}{m!}\; x^{rm}.
\end{aligned}
\end{equation}
It is easy to see from Stirling's formula that 
the auxiliary functions are absolutely 
convergent with the radius of convergence
$e^{-\frac{1}{r}}$. 
Since these functions do not have any constant
terms, we have
\begin{equation}
\label{eq:xi-diff}
\xi_{\ell+1}^{r,k}(x)=x\frac{d}{dx}\;
\xi_{\ell}^{r,k}(x), 
\qquad k = 0,1,\dots,r-1.
\end{equation}
Therefore, all we need is to find the functions
at $\ell = -1$.
The standard procedure to compute (\ref{eq:xi})
is to use the \emph{Lambert function}.
Let us define
\begin{equation}
\label{eq:y}
y(x) = \xi_{-1}^{1,0}(x) =
\sum_{m=1}^\infty 
\frac{m^{m-1}}{m!}\; x^{m}.
\end{equation}
Then its inverse  is given by
the Lambert function (\ref{eq:Lambert}),
which can be easily checked by the
Lagrange inversion formula,
and the following formula holds for every
complex number $\a\in\bC^*$ (see 
for example, \cite{CGHJK}):
\begin{equation}
\label{eq:alpha-Lambert}
\exp\big(\a y(x)\big) = \sum_{m=0}^\infty \frac{\a 
(m+\a)^{m-1}}{m!} x^m .
\end{equation}
Therefore, the base case for 
(\ref{eq:xi}) is computed by
\begin{equation}
\begin{aligned}
\label{eq:xi-1 in x}
\xi_{-1}^{r,k}(x) 
&= 
\frac{1}{k} x^k 
\exp\left(\frac{k}{r}y(rx^r)
\right), \qquad k\ne 0,
\\
\xi_{-1}^{r,0}(x) 
&= 
\frac{1}{r} y(rx^r).
\end{aligned}
\end{equation}
We now \emph{define} the variable $z$ by
\begin{equation}
\label{eq:z}
z = z(x) = \left(\frac{1}{r} y(rx^r)\right)
^{\frac{1}{r}},
\end{equation}
so that its inverse function is given by
 the $r$-Lambert 
curve $x = z e^{-z^r}$ (\ref{eq:r-Lambert}).
In terms of $z$, the auxiliary functions 
(\ref{eq:xi-1 in x}) take much simpler form
\begin{equation}
\begin{aligned}
\label{eq:xi-1 in z}
\xi_{-1}^{r,k}(x) 
&= 
\frac{1}{k}z^k, \qquad k\ne 0,
\\
\xi_{-1}^{r,0}(x) 
&= 
z^r.
\end{aligned}
\end{equation}
The differential operator of (\ref{eq:xi-diff})
in $z$ is
\begin{equation}
\label{eq:dz}
x\frac{d}{dx}= \frac{z}{1-rz^r}\;\frac{d}{dz}.
\end{equation}
Since $F_{0,1}^{(r)}= \xi_{-2}^{r,0}(x)$, we have
$$
\frac{z}{1-rz^r}\;\frac{d}{dz} F_{0,1}^{(r)}
= \xi_{-1}^{r,0}(x) = z^r.
$$
Therefore, considering the fact that 
$z=0 \Longrightarrow x=0\Longrightarrow 
F_{0,1}^{(r)}=0$, 
 we find
\begin{equation*}
F_{0,1}^{(r)}(z) = \frac{1}{r}z^r-\half z^{2r},
\end{equation*}
which proves (\ref{eq:F01}).
Then from (\ref{eq:W01=dF01}), we have
\begin{align*}
dF_{0,1}^{(r)}(z) &= z^{r-1}(1-rz^r)dz , 
\\
yd\log(x) &=
yz^{-1}e^{z^r}
d\left(ze^{-z^r}\right) = yz^{-1}(1-rz^r)dz.
\end{align*}
Hence 
$$
y = f(z) = z^{r}.
$$
We have thus  determined the Lagrangian 
immersion
\begin{equation}
\label{eq:Lagrangian immersion}
\iota:\Sigma =  \bC^*\lrar T^*\bC^*, \qquad
\begin{cases}
x = ze^{-z^r}\\
y = z^r,
\end{cases}
\qquad z\in \Sigma,
\end{equation}
in agreement with (\ref{eq:parametric}).
We note that (\ref{eq:Lagrangian immersion}) 
implies
$
rx^r = (ry) e^{-ry},
$
hence $y(rx^r) = ry =rz^r$, which is consistent
with (\ref{eq:z}).

Another important feature of the Eynard-Orantin
theory \cite{MS} is the special relation 
between  the Laplace transform
$F_{0,2}^{(r)}(z_1,z_2)$
of the \emph{annulus amplitude} 
$H_{0,2}^{(r)}(\mu_1,\mu_2)$ 
 and the difference of the Riemann's prime 
 forms of the $x$-projection $\pi:\Sigma\lrar \bC^*$
 \cite{EO1}.  Again from \cite{JPT}, we know
 the annulus amplitude of the orbifold Hurwitz 
 numbers:
 $$
 H_{0,2}^{(r)}(\mu_1,\mu_2) 
 = r^{\la\frac{\mu_1}{r}\ra+\la\frac{\mu_2}{r}
 \ra}\cdot
 \frac{1}{\mu_1+\mu_2}\cdot
 \frac{\mu_1^{\lfloor \frac{\mu_1}{r}\rfloor}}
 {\lfloor \frac{\mu_1}{r}\rfloor!}
 \cdot
 \frac{\mu_2^{\lfloor \frac{\mu_2}{r}\rfloor}}
 {\lfloor \frac{\mu_2}{r}\rfloor!},
 \qquad \text{if}\;\;\mu_1+\mu_2\equiv 0 \mod r,
 $$
and $H_{0,2}^{(r)}(\mu_1,\mu_2) =0$ otherwise.
Here  
$
\la q\ra = q-\lfloor q\rfloor
$
is the fractional part of $q\in\bQ$.

\begin{proof}[Proof of \rm{(\ref{eq:F02})}]
Write $\mu_i=rm_i+k_i$, $i=1,2$, with 
$0\le k_i\le r-1$. Then
$$
\mu_1+\mu_2 \equiv 0 \mod r \Longleftrightarrow
\begin{cases}
k_1=k_2 = 0 \quad {\text{or}}
\\
k_1+k_2=r.
\end{cases}
$$
Therefore, we obtain a partial differential 
equation
\begin{align*}
&\left(
\frac{z_1}{1-rz_1}\frac{\partial}{\partial z_1}
+
\frac{z_2}{1-rz_2}\frac{\partial}{\partial z_2}
\right)
F_{0,2}^{(r)}(z_1,z_2) 
\\
= &
\left(
x_1\frac{\partial}{\partial x_1}
+
x_2\frac{\partial}{\partial x_2}
\right)
 \sum_{(\mu_1,\mu_2)\in\bZ_+^2}
 H_{0,2}^{(r)}(\mu_1,\mu_2)
 x_1^{\mu_1}x_2^{\mu_2}
 \\
 =&
 \sum_{(\mu_1,\mu_2)\in\bZ_+^2}
 r^{\la\frac{\mu_1}{r}\ra+\la\frac{\mu_2}{r}
 \ra}\cdot
 \frac{\mu_1^{\lfloor \frac{\mu_1}{r}\rfloor}}
 {\lfloor \frac{\mu_1}{r}\rfloor!}
 \cdot
 \frac{\mu_2^{\lfloor \frac{\mu_2}{r}\rfloor}}
 {\lfloor \frac{\mu_2}{r}\rfloor!} x_1^{\mu_1}x_2^{\mu_2}
 \\
 =&
 \left(
 \sum_{m_1=1}^\infty 
 \frac{(rm_1)^{m_1}}{m_1!}x_1^{rm_1}
 \right)
 \left(
 \sum_{m_1=2}^\infty 
 \frac{(rm_2)^{m_2}}{m_2!}x_2^{rm_2}
 \right)
 \\
 &+
 r\sum_{k=1} ^{r-1} 
 \left(
 \sum_{m_1=0}^\infty 
 \frac{(rm_1+k)^{m_1}}{m_1!}x_1^{rm_1+k}
 \right)
 \left(
 \sum_{m_2=0}^\infty 
 \frac{(rm_2+r-k)^{m_2}}{m_2!}x_2^{rm_1+r-k}
 \right)
 \\
 =&
 \xi_0 ^{r,0}(x_1)\xi_0^{r,0}(x_2)
 +
 r\sum_{k=1} ^{r-1} \xi_0 ^{r,k}(x_1)
 \xi_0^{r,r-k}(x_2)
 \\
 =&
 \frac{rz_1^r}{1-rz_1}\cdot
 \frac{rz_2^r}{1-rz_2}
 +
 r \frac{1}{1-rz_1}\cdot
 \frac{1}{1-rz_2}
 \sum_{k=1}^{r-1}
 z_1^kz_2^{r-k}
 \\
 =&
 \frac{1}{(1-rz_1)(1-rz_2)}
 \left(
 r^2z_1^rz_2^r+r\frac{z_1z_2^r-z_2z_1^r}{z_2-z_1}
 \right),
\end{align*}
where we have used (\ref{eq:xi-diff}) to find
the auxiliary functions. 
It is easy to check that 
$$
 F_{0,2}^{(r)}(z_1,z_2) =
 \log\frac{z_1-z_2}{x_1-x_2}-(z_1^r+z_2^r)
 $$
 is a solution of this
 differential equation, where $x_i = z_i e^{-z_i^r}$.
We note that as a convergent
power series in $(x_1,x_2)$, 
$F_{0,2}^{(r)}$ does not have any constant
term. Since the convergent series
eigenfunctions of
the Euler differential operator 
$x_1\frac{\partial}{\partial x_1}
+
x_2\frac{\partial}{\partial x_2}
$
are homogeneous polynomials, its kernel 
consists of constants. Therefore, (\ref{eq:F02})
is the only solution that satisfies the initial 
condition. 
\end{proof}

The main structural 
difference between the cut-and-join
equation (\ref{eq:CAJ}) and the 
differential recursion (\ref{eq:diffrecursion})
is whether the unstable geometries are
included in the right-hand side or not. 
While (\ref{eq:diffrecursion}) is a genuine
recursion for $F_{g,n}^{(r)}$ with respect
to $2g-2+n$, (\ref{eq:CAJ}) only gives a relation 
because $H_{g,n}^{(r)}$ appears on each side of
the equation.  In proving (\ref{eq:diffrecursion}), 
we first calculate the Laplace transform of the
cut-and-join equation, then use 
(\ref{eq:F01}) and (\ref{eq:F02}) to eliminate 
the unstable geometries from the right-hand side.

\begin{lem}
\label{eq:LT of CAJ in x}
The straightforward Laplace transform
of the cut-and-join equation \rm{(\ref{eq:CAJ})}
gives a differential equation
\begin{multline}
\label{eq:LT CAJ}
\left(2g-2+n+\frac{1}{r}\sum_{i=1}^n x_i
\frac{\partial}{\partial x_i}\right)
F_{g,n}^{(r)}[x_1,\dots,x_n]
\\
=\half \sum_{i\ne j}\frac{1}{x_i-x_j} 
\left(
x_i^2\frac{\partial}{\partial x_i}
F_{g,n-1}^{(r)}\big[x_{[\hat{j}]}\big]
-
x_j^2\frac{\partial}{\partial x_j}
F_{g,n-1}^{(r)}\big[x_{[\hat{i}]}\big]
\right)
-
\sum_{i\ne j} 
x_i\frac{\partial}{\partial x_i}
F_{g,n-1}^{(r)}\big[x_{[\hat{j}]}\big]
\\
+\half \sum_{i=1}^n 
\left.
u_1\frac{\partial}
{\partial u_1}u_2\frac{\partial}{\partial u_2}
F_{g-1,n+1}^{(r)}
\big[u_1,u_2,x_{[\hat{i}]}
\big]
\right|_{u_1=u_2=x_i}
\\
+\half \sum_{i=1}^n 
\sum_{\substack{g_1+g_2=g\\I\sqcup J=[\hat{i}]}}
\left(
x_i\frac{\partial}{\partial x_i}
F_{g_1,|I|+1}^{(r)}[x_i,x_I]
\right)
\left(
x_i\frac{\partial}{\partial x_i}
F_{g_2,|J|+1}^{(r)}[x_i,x_J]
\right),
\end{multline}
where $F_{g,n}^{(r)}[x_1,\dots,x_n]$ is defined
by \rm{(\ref{eq:Fgn in x})}.
\end{lem}

\begin{proof}
Since the cut-and-join equation (\ref{eq:CAJ})
has the 
same structure as the simple Hurwitz number case
of \cite{MZ}, the calculation of the Laplace
transform goes exactly in parallel. Therefore, 
the left-hand side of (\ref{eq:LT CAJ}) and the
second and the third lines of the right-hand side
are immediate from (\ref{eq:Fgn in x}) and
(\ref{eq:CAJ}), noting how 
$x_i\frac{\partial}{\partial x_i}$ acts on 
$x_i^{\a}$. 

The trick we need is
$$
\sum_{\mu_1,\mu_2\ge 0}f(\mu_1+\mu_2)
x_1^{\mu_1}x_2^{\mu_2}
=
\sum_{k=0}^\infty f(k)
\sum_{\mu_1+\mu_2 = k} 
x_1^{\mu_1}x_2^{\mu_2}
=
\sum_{k=0}^\infty f(k)
\left(
\frac{x_1^{k+1}-x_2^{k+1}}{x_1-x_2}
\right),
$$
which is valid if the  series on the left-hand side 
is absolutely convergent. 
The ELSV formula (\ref{eq:ELSV}) tells us
that the power series 
$F_{g,n}^{(r)}[x_1,\dots,x_n]$
is convergent on the polydisk
$$
\big(|x_1|<e^{-\frac{1}{r}}\big)
\times
\cdots\times
\big(|x_n|<e^{-\frac{1}{r}}\big).
$$
Therefore, we can compute the Laplace transform
of the first line of the right-hand side of 
(\ref{eq:CAJ}) as follows. 
\begin{align*}
&\half \sum_{\mu_1,\dots,\mu_n\in \bZ_+}
\sum_{i\ne j}(\mu_i+\mu_j)H_{g,r-1}^{(r)}
\left(\mu_i+\mu_j,\mu_{[\hat{i},\hat{j}]}\right)
\prod_{i=1}^n x_i^{\mu_i}
\\
=&
\half \sum_{i\ne j}\sum_{\nu=0}^\infty
\sum_{\vec{\mu}_{[\hat{i},\hat{j}]}\in
\bZ_+^{n-2}}
\sum_{\mu_i+\mu_j=\nu}
\nu H_{g,r-1}^{(r)}
\left(\nu,\mu_{[\hat{i},\hat{j}]}\right)
\prod_{i=1}^n x_i^{\mu_i}
\\
&\qquad -
\half \sum_{i\ne j}\sum_{\nu=0}^\infty
\sum_{\vec{\mu}_{[\hat{i},\hat{j}]}\in
\bZ_+^{n-2}}
\nu H_{g,r-1}^{(r)}
\left(\nu,\mu_{[\hat{i},\hat{j}]}\right)
(x_i^\nu+x_j^\nu)
\prod_{k\ne i,j}^n x_k^{\mu_k}
\\
=&
\half \sum_{i\ne j}\sum_{\nu=0}^\infty
\sum_{\vec{\mu}_{[\hat{i},\hat{j}]}\in
\bZ_+^{n-2}}
\nu H_{g,r-1}^{(r)}
\left(\nu,\mu_{[\hat{i},\hat{j}]}\right)
\frac{x_i^{\nu+1}-x_j^{\nu+1}}{x_i-x_j}
\prod_{k\ne i,j}^n x_k^{\mu_k}
\\
&\qquad 
-\sum_{i\ne j}
x_i\frac{\partial}{\partial x_i}
F_{g,r-1}^{(r)}
\left[x_i,x_{[\hat{i},\hat{j}]}
\right]
\\
=&
\half \sum_{i\ne j}
\frac{1}{x_i-x_j}
\left(
x_i^2\frac{\partial}{\partial x_i}
F_{g,r-1}^{(r)}
\left[x_i,x_{[\hat{i},\hat{j}]}
\right]
-
x_j^2\frac{\partial}{\partial x_j}
F_{g,r-1}^{(r)}
\left[x_j,x_{[\hat{i},\hat{j}]}
\right]
\right)
\\
&\qquad 
-\sum_{i\ne j}
x_i\frac{\partial}{\partial x_i}
F_{g,r-1}^{(r)}
\left[x_{[\hat{j}]}
\right].
\end{align*}
This completes the proof.
\end{proof}

\begin{proof}[Proof of Theorem~\ref{thm:Fgn recursion}]
The
conversion of (\ref{eq:LT CAJ})
to the form (\ref{eq:diffrecursion})
is now straightforward,
using (\ref{eq:dz}), and 
substituting the unstable geometries with the
actual values
(\ref{eq:F01}) and (\ref{eq:F02}) in 
the right-hand side. 

The contribution from the 
terms of $g_1=0, I=\emptyset$ and
$g_2=0,J=\emptyset$ in the third line of 
the right-hand side of (\ref{eq:LT CAJ})
is 
$$
\sum_{i=1}^n
\left(
x_i\frac{\partial}{\partial x_i} F_{0,1}^{(r)}[x_i]
\right)
\left(
x_i\frac{\partial}{\partial x_i} 
F_{g,n}^{(r)}\left[x_i, x_{[\hat{i}]}\right]
\right)
=
\sum_{i=1}^n
z_i^r \frac{z_i}{1-rz_i^r}\frac{\partial}{\partial z_i}
F_{g,n}^{(r)}(z_1,\dots,z_n).
$$
If we bring this term to the left-hand side
or (\ref{eq:LT CAJ}), then
we have
\begin{multline*}
\left(2g-2+n+\frac{1}{r}\sum_{i=1}^n x_i
\frac{\partial}{\partial x_i}\right)
F_{g,n}^{(r)}(z_1,\dots,z_n)
-
\sum_{i=1}^n
z_i^r \frac{z_i}{1-rz_i^r}\frac{\partial}{\partial z_i}
F_{g,n}^{(r)}(z_1,\dots,z_n)
\\
=
\left(2g-2+n+\frac{1}{r}\sum_{i=1}^n 
\frac{z_i}{1-rz_i^r}(1-rz_i^r)
\frac{\partial}{\partial z_i}
\right)
F_{g,n}^{(r)}(z_1,\dots,z_n),
\end{multline*}
which is the left-hand side of 
(\ref{eq:diffrecursion}).

The other unstable terms come from
$g_1=0, I=\{j\}$ and $g_2=0, J=\{j\}$. The 
contribution is
\begin{multline*}
\sum_{i=1}^n\sum_{j\ne i}
\left(
x_i\frac{\partial}{\partial x_i} F_{0,2}^{(r)}[x_i,x_j]
\right)
\left(
x_i\frac{\partial}{\partial x_i} 
F_{g,n-1}^{(r)}\left[x_i, x_{[\hat{i}, \hat{j}]}\right]
\right)
\\
=
\sum_{i\ne j}
\left(
\frac{z_i}{1-rz_i^r}
\frac{\partial}{\partial z_i}\left(\log(z_i-z_j)
-(z_i^r+z_j^r)
\right)
-x_i\frac{\partial}{\partial x_i}\log(x_i-x_j)
\right)
\left(
x_i\frac{\partial}{\partial x_i} 
F_{g,n-1}^{(r)}\left[x_{[\hat{j}]}\right]
\right)
\\
=
\sum_{i\ne j}
\left(
\frac{z_i}{1-rz_i^r}
\left(\frac{1}{z_i-z_j}
-rz_i^{r-1}
\right)
-\frac{x_i}{x_i-x_j}
\right)
\left(
x_i\frac{\partial}{\partial x_i} 
F_{g,n-1}^{(r)}\left[x_{[\hat{j}]}\right]
\right).
\end{multline*}
These terms and the first line of the 
right-hand side of (\ref{eq:LT CAJ}) together
yield
\begin{multline*}
\sum_{i\ne j}
\left(
\frac{z_i}{1-rz_i^r}
\left(\frac{1}{z_i-z_j}
-rz_i^{r-1}
\right)
-\frac{x_i}{x_i-x_j}+\frac{x_i}{x_i-x_j}-1
\right)
\left(
x_i\frac{\partial}{\partial x_i} 
F_{g,n-1}^{(r)}\left[x_{[\hat{j}]}\right]
\right)
\\
=
\sum_{i\ne j}
\left(
\frac{z_i}{1-rz_i^r}
\left(\frac{1}{z_i-z_j}
-rz_i^{r-1}
\right)
-1
\right)
\left(
\frac{z_i}{1-rz_i^r}\frac{\partial}{\partial z_i} 
F_{g,n-1}^{(r)}\left(z_{[\hat{j}]}\right)
\right)
\\
=
\sum_{i\ne j}
\frac{z_iz_j}{z_i-z_j}
\frac{1}{(1-rz_i^r)^2}\frac{\partial}{\partial z_i} 
F_{g,n-1}^{(r)}\left(z_{[\hat{j}]}\right),
\end{multline*}
which is the same as the first line of the 
right-hand side of (\ref{eq:diffrecursion}).

Converting the second line 
and the stable terms in the third line
of the right-hand side
 of (\ref{eq:LT CAJ}) is
straightforward. 
This completes the proof
of Theorem~\ref{thm:Fgn recursion}.
\end{proof}

\section{Some properties of the free energies}
\label{sect:properties}

In this section we derive some properties of the
free energies and compute a few examples.
We also check our results with 
closed formulas obtained in \cite{JPT}.

A direct consequence of the ELSV formula
(\ref{eq:ELSV}) is the following.

\begin{prop}
\label{prop:tensor}
The Laplace transform of \rm{(\ref{eq:ELSV})}, 
 the free energy of type $(g,n)$, 
is an element of the tensor algebra
\begin{equation}
\label{eq:tensor}
F_{g,n}^{(r)}(z_1,\dots,z_n)
\in \Sym^{\tensor n}\left(\bC(\bP^1)\right),
\end{equation}
except for $F_{0,2}^{(r)}(z_1,z_2)$. 
The poles $F_{g,n}^{(r)}(z_1,\dots,z_n)$
are located at
$$
D\times \big(\bP^1\big)^{n-1}
\cup \bP^1\times D\times \cdots\times \bP^1
\cup\cdots\cup \big(\bP^1\big)^{n-1}\times D,
$$
where 
\begin{equation*}
D = \{z\in\bC\;|\;1-rz^r=0\}.
\end{equation*}
The highest total degree of poles of 
$F_{g,n}^{(r)}$ is $6g-6+3n$.
\end{prop}

\begin{proof}
This follows from the fact that the coefficient
$$
r^{1-g +\sum_{i=1}^n \la \frac{\mu_i}{r}\ra}
\int_{\Mbar_{g,-\vec{\mu}}(BG)}
\prod_{i=1}^n \psi_i^{d_i}\sum_{j\ge 0} 
(-r)^j\lam_j
$$
of (\ref{eq:ELSV}) depends only on $\vec{\mu}
\mod r$, hence
\begin{multline}
\label{eq:Fgn Laplace}
F_{g,n}^{(r)}(z_1,\dots,z_n)
\\
=
\sum_{\vec{\mu}\in\bZ_+^n, 
\sum \mu_i\equiv 0\; (r)}
r^{1-g +\sum_{i=1}^n \la \frac{\mu_i}{r}\ra}
\int_{\Mbar_{g,-\vec{\mu}}(BG)}
\frac{\sum_{j\ge 0}(-r)^j \lam_j}
{\prod_{i=1}^n(1-\mu_i\psi_i)}
\prod_{i=1}^n
 \frac{\mu_i^{\lfloor\frac{\mu_i}{r}\rfloor}}
{\lfloor\frac{\mu_i}{r}\rfloor!}
x_i^{\mu_i}
\\
=
\sum_{\substack{0\le k_1,\dots,k_n< r,
\sum k_i\equiv 0\; (r)\\
d_1+\cdots+d_n\le 3g-3+n}}
r^{1-g +\sum_{i=1}^n \frac{k_i}{r}}
\left(
\int_{\Mbar_{g,-\vec{k}}(BG)}
\prod_{i=1}^n \psi_i^{d_i}
\sum_{j\ge 0}(-r)^j \lam_j
\right)
\prod_{i=1}^n \xi_{d_i}^{r,k_i}(x_i).
\end{multline}
For $d_i\ge 0$, each $\xi_{d_i}^{r,k_i}(x_i)$ 
is a rational function
in $z_i$ with poles at $z_i\in D$ of degree
$2d_i+1$ due to (\ref{eq:xi-diff}), 
(\ref{eq:xi-1 in z}), and (\ref{eq:dz}).
The highest degree poles occur when
$d_1+\cdots+d_n=3g-3+n$, and then
$F_{g,n}^{(r)}$ has poles of degree $6g-6+3n$.
\end{proof}

Using the same notation as in 
Theorem~\ref{thm:JPT}, let us denote
\begin{equation}
\label{eq:one point}
\la \tau_{2g-2+j}\lam_{g-j}\ra^{(r)}
=
\int_{\Mbar_{g,1}(BG)} \psi_1^{2g-2+j}
\lam_{g-j},
\end{equation}
where $G=\bZ/r\bZ$.
The generating function of these one-point
intersection numbers is determined in 
\cite{JPT}:
\begin{equation}
\label{eq:JPT one-point}
\frac{1}{2r}\;
\left(
\frac{r\hbar/2}{\sin(r\hbar/2)}
\right)^u
\frac{1}{\sin(\hbar/2)}
=
\frac{1}{r\hbar}+\sum_{g=1}^\infty 
\left(\sum_{j=0}^g
\la \tau_{2g-2+j}\lam_{g-j}\ra^{(r)}
u^j
\right)
 \hbar^{2g-1}.
\end{equation}
Note that from (\ref{eq:ELSV}) and (\ref{eq:xi})
we can calculate the one-point free energies:
\begin{equation}
\label{eq:Fg1}
F_{g,1}^{(r)}(z)
=
\sum_{j=0}^g
(-1)^{g-j} r^{1-j} 
\la \tau_{2g-2+j}\lam_{g-j}\ra^{(r)}
\xi_{2g-2+j}^{r,0}(x).
\end{equation}
For example, in terms of 
\begin{equation}
\label{eq:t}
t = \frac{1}{1-rz^r},
\end{equation}
we have
\begin{align*}
F_{1,1}^{(r)}(z) &= \frac{1}{24}
 (r^2t^3-r^2t^2-t+1),
\\
F_{2,1}^{(r)}(z) &= 
\frac{r^2}{5760}
\bigg(
525r^4 t^9 -1575r^4t^8
+10r^2(167r^2-15)t^7
+350r^2(-2r^2+1)t^6 
\\
&\qquad 
+(68r^4-260r^2+21)t^5
+(12r^4+60r^2-35)t^4
+14t^3
\bigg),
\\
F_{3,1}^{(r)}(z)&=
\frac{r^4}{2903040}
\bigg(
4729725 r^6 t^{15}
-23648625 r^6 t^{14}
+ 24255 r^4(2012 r^2-45) t^{13} 
\\
&\qquad
+ 35035 r^4 (- 1516 r^2+135) t^{12}
+ 35 r^2 (914912 r^4 - 235116 r^2+3969) t^{11} 
\\
&\qquad
+ 231 r^2 (-43156 r^4+ 31430 r^2-2205) t^{10}
\\
&\qquad
+ 35 (31016 r^6- 95340 r^4+ 20580 r^2-279 ) 
t^9 
\\
&\qquad
+ 7 (15416 r^6+ 100596 r^4 - 69384 r^2 +4185) 
t^8
\\
&\qquad
+ 12 (-1128 r^6- 2646 r^4+12789 r^2-2635) t^7
 \\
&\qquad
 + 6 (-320 r^6-840 r^4-2940 r^2+2387) t^6 
  -2232 t^5
\bigg).
\end{align*}
In general,

\begin{prop}
\label{prop:Fg1 polynomial}
The one-point free energy 
 $F_{g,1}^{(r)}(z)$ of genus $g$ is
a polynomial of degree $6g-3$ in 
$t = \frac{1}{1-rz^r}$.
\end{prop}
\begin{proof}
The expression (\ref{eq:Fg1}) tells us that 
$F_{g,1}^{(r)}(z)$ is a function in $z^r$. More 
precisely, it is a ratio of a polynomial in 
$z^r$ and a power of $1-rz^r$. Therefore, it is 
a Laurent polynomial in $t$. 
The only auxiliary functions appearing in 
(\ref{eq:Fg1}) are $\xi_{\ell}^{r,0}(x)$ for $\ell\ge 0$. 
From (\ref{eq:xi-diff}) and (\ref{eq:xi-1 in z}) we
calculate
\begin{equation}
\label{eq:xi_0}
\begin{aligned}
\xi_0^{r,0}(x) &= \frac{rz^r}{1-rz^r} = t-1,
\\
\xi_{\ell}^{r,0}(x) &= 
\left(
rt^2(t-1)\frac{d}{dt}
\right)^\ell (t-1),
\end{aligned}
\end{equation}
since 
$$
x\frac{d}{dx} = rt^2(t-1)\frac{d}{dt}.
$$
Therefore, $F_{g,1}^{(r)}(z)$ is a polynomial
of degree $2(3g-2)+1$ in $t$.
The degree of the polynomial is the same as the
degree of poles of Proposition~\ref{prop:tensor}
for $n=1$.
\end{proof}

The initial cases of the differential recursion
(\ref{eq:diffrecursion}) are $(g,n)=(1,1)$
and $(0,3)$. For the $g=n=1$ case, the differential
equation is
$$
\left(
1+\frac{1}{r} z \frac{\partial}{\partial z}
\right)F_{1,1}^{(r)}(z)
= \left.\half \frac{z^2}{(1-rz^r)^2} 
\frac{\partial^2}{\partial u_1\partial u_2}
F_{0,2}^{(r)}(u_1,u_2)
\right|_{u_1=u_2=z}.
$$
The unique solution to this equation 
with the initial condition $F_{1,1}^{(r)}(0)=0$
agrees with the above computation
using the result of \cite{JPT}.

The free energy
$F_{0,3}^{(r)}(z_1,z_2,z_3)$ can also be 
calculated  from (\ref{eq:ELSV}) since
$\Mbar_{0,3}$ is a point. Thus the $BG$ Hodge
integral contribution in the formula is simply 
$1$. We have
\begin{multline*}
F_{0,3}^{(r)}(z_1,z_2,z_3)
=
\sum_{\substack{\vec{\mu}\in\bZ_+^3\\
\mu_1+\mu_2+\mu_3 \equiv 0}}
r^{1+\la\frac{\mu_1}{r}\ra+\la\frac{\mu_2}{r}\ra
+\la\frac{\mu_3}{r}\ra}
\prod_{i=1}^3 \frac{\mu_i^{\lfloor \frac{\mu_i}{r}
\rfloor}}
{\lfloor \frac{\mu_i}{r}\rfloor !}x_i^{\mu_i}
\\
=
r \xi_0 ^{r,0}(x_1)\xi_0 ^{r,0}(x_2)\xi_0 ^{r,0}(x_3)
+
r^2 \sum_{\substack{k_1+k_2+k_3=r\\
0\le k_i\le r-1}}
\xi_0 ^{r,k_1}(x_1)\xi_0 ^{r,k_2}(x_2)
\xi_0 ^{r,k_3}(x_3)
\\
+
r^3 \sum_{\substack{k_1+k_2+k_3=2r\\
0\le k_i\le r-1}}
\xi_0 ^{r,k_1}(x_1)\xi_0 ^{r,k_2}(x_2)
\xi_0 ^{r,k_3}(x_3).
\end{multline*}
More concretely, 
\begin{align*}
F_{0,3}^{(2)}(z_1,z_2,z_3)&=
8\frac{ z_1 z_2 z_3 (z_1 + z_2 + z_3 + 2 z_1 z_2 z_3)}
{(1 - 2 z_1^2) (1 - 2 z_2^2) (1-  2 z_3^2)},
\\
F_{0,3}^{(3)}(z_1,z_2,z_3)
&=
9\frac{z_1 z_2 z_3 \left(1 +3z_1z_2z_3+ 3
\sum_{i\ne j}z_i^2z_j +9z_1^2z_2^2z_3^2\right)}
   {(1 - 3 z_1^3) (1 -  3 z_2^3) (1 - 3 z_3^3)}.
\end{align*}

\section{The quantum curve}
\label{sect:quantum}

Since the $r$-Lambert curve (\ref{eq:r-Lambert})
has genus $0$, we define the partition function
$Z(z,\hbar)$ as in (\ref{eq:partition}).
In this section we prove 
Theorem~\ref{thm:quantum curve}.

\begin{prop}
\label{prop:Fgn(z) polynomial}
The principal specialization
$F_{g,n}^{(r)}(z,\dots,z)$ for $2g-2+n>0$
is a polynomial in $t$ of degree $6g-6+3n$,
where $t$ is the variable introduced in
\rm{(\ref{eq:t})}. 
\end{prop}
\begin{proof}
This is an immediate consequence of
 Proposition~\ref{prop:tensor}
and its proof. 
\end{proof}
For unstable geometries, we use the same 
argument of \cite{MS} to find
\begin{align}
\label{eq:F01(t)}
F_{0,1}^{(r)}(z) &= \frac{1}{2r^2}
\left(1-\frac{1}{t^2}\right),
\\
\label{eq:F02(t)}
F_{0,2}^{(r)}(z,z) &= \frac{1}{r}
\left(1-\frac{1}{t}\right)+\log t.
\end{align}

\begin{prop}
\label{prop:Sm}
The $1$-variable functions
\begin{equation}
\label{eq:Sm}
S_m^{(r)}(z) = \sum_{2g-2+n=m-1}\frac{1}{n!}\;
F_{g,n}^{(r)}(z,\dots,z), \qquad m=0, 1, 2, \dots,
\end{equation}
satisfy the second order ordinary differential
equation
\begin{equation}
\label{eq:Sm ODE}
\begin{aligned}
  \left(m+ \frac{1}{r}\frac{d}{dz}\right) 
  S_{m+1}^{(r)}(z) 
  &= 
  \frac12 \left[\frac{d}{dz}
  \left(\frac{z}{1-rz^r}\right)^2
  {-\frac{2z}{(1-rz^r)^2}}\right]\cdot \frac{d}{dz}
  S_m^{(r)}(z) 
  \\
  &+
  \frac12\frac{z^2}{(1-rz^r)^2}
  \left(\frac{d^2}{dz^2}S_m^{(r)}(z)
    +\sum_{\substack{a+b=m+1\\a,b\ge2}}
    \frac{d}{dz}S_a^{(r)}(z)\frac{d}{dz}
    S_b^{(r)}(z)\right).
\end{aligned}
\end{equation}
\end{prop}

\begin{proof}
The principal specialization of the differential
recursion 
(\ref{eq:diffrecursion}) reduces to the following
ordinary differential equation.
\begin{align*}
 \left(2g-2+n + \frac{1}{r}z\frac{d}{dz}\right) 
 &
 F^{(r)}_{g,n}(z,\cdots, z) 
 =
 \frac{n}2 \frac{d}{dz}\left(\frac{z^2}{(1-rz^r)^2}\right)\frac{d}{dz}
 F^{(r)}_{g,n-1}(z,\cdots,z)
 \\
&{-n\frac{z}{(1-rz^r)^2}\frac{d}{dz}
F_{g,n-1}^{(r)}(z,\cdots,z)}
\\
&+
 \frac12 n(n-1) \frac{z^2}{(1-r z^r)^2}\frac{\partial^2}{\partial u^2}\bigg|_{u=z}
F^{(r)}_{g,n-1}(z,\cdots,z)
\\
& + 
\frac{n}2 \frac{z^2}{(1-r z^r)^2}\frac{\partial^2}{\partial u_1\partial u_2}\bigg|_{u_1=u_2=z}
F^{(r)}_{g-1,n+1}(u_1,u_2,z,\cdots,z)
\\
&+
\frac{n!}{2}\sum_{\substack{g_1+g_2 = g\\n_1 + n_2 = n-1}}^{\text{stable}}
\Bigg[
\frac{1}{(n_1+1)!}\frac{z}{1-r z^r}\frac{d}{dz} F_{g_1,n_1+1}^{(r)}(z,\cdots,z) 
\\
&\qquad+
\frac{1}{(n_2+1)!}\frac{z}{1-r z^r}\frac{d}{dz} F_{g_2,n_2+1}^{(r)}(z,\cdots,z)
\Bigg].
\end{align*}
The summation (\ref{eq:Sm}) proves the 
proposition.
\end{proof}

\begin{proof}[Proof of Theorem~\ref{thm:quantum curve}]
Note that we have
$$
  S_0^{(r)}(z) = z^r\left(\frac1r -\frac12 z^r\right), \quad
  S_1^{(r)}(z) = -\frac12 \log(1-rz^r) -\frac12 z^r.
$$
If we include these unstable terms into
(\ref{eq:Sm ODE}), then we obtain
\begin{multline*}
  \left(m + \frac1r\frac{z}{1-rz^r}\frac{d}{dz}\right)S_{m+1}^{(r)}(z)  
  \\
   = \frac12 \left(\left(\frac{z}{1-rz^r}\frac{d}{dz}\right)^2 S_m^{(r)}(z) +
    \sum_{a+b=m+1}\frac{z}{1-rz^r}\frac{d}{dz} S_a^{(r)}(z)\cdot\frac{z}{1-rz^r}\frac{d}{dz}
    S_b^{(r)}(z)\right)
    \\
  -{\frac{1}{2}\frac{z}{1-rz^r}}\frac{d}{dz}S_m^{(r)}(z).
\end{multline*}
In terms of the generating series
\begin{displaymath}
  F^{(r)}(z,\hbar) = \sum_{m=0}^\infty S_m^{(r)} \hbar^{m-1},
\end{displaymath}
the equation becomes
\begin{align*}
  &\hbar\frac{\partial}{\partial\hbar} F^{(r)}(z,\hbar) + \frac1r \frac{z}{1-rz^r}\frac{d}{dz}
  F^{(r)}(z,\hbar)\\
& = \frac{\hbar}{2}\left[\left(\frac{z}{1-rz^r}\frac{d}{dz}\right)^2 F^{(r)}(z,\hbar) +
  \left(\frac{z}{1-rz^r}\frac{d}{dz} F^{(r)}(z,\hbar)\right)^2\right]\\
&-{\frac\hbar 2 \frac{z}{1-rz^r}}\frac{d}{dz}F^{(r)}(z,\hbar).
\end{align*}
Since $Z^{(r)}(z,\hbar) = \exp F^r(z,\hbar)$
and $D=
x\frac{d}{dx} = \frac{z}{1-rz^r}\frac{d}{dz}$, 
we have
\begin{equation}\label{eq:ps}
  \left(\dfrac{\partial}{\partial\hbar} + \left(\dfrac{1}{r\hbar} +{\frac12}\right) x\dfrac{d}{dx} -
    \dfrac{1}{2}\left(x\dfrac{d}{dx}\right)^2\right) Z^{(r)}(z,\hbar) = 0,
\end{equation}
which establishes (\ref{eq:Q}).

Now define
\begin{align}
\label{eq:operator P}
P&= \hbar\frac{\partial}{\partial w}
+
e^{-\frac{r-1}{2}\hbar\frac{\partial}{\partial
    w}}e^{-rw} e^{\frac{r-1}{2}\hbar\frac{\partial}{\partial w}}e^{-r\hbar\frac{\partial}{\partial
    w}}
    \\
    Q &= 
    \frac{\hbar}{2}\frac{\partial^2}{\partial w^2}
    +\left(\frac1r +\frac\hbar
2\right)\frac{\partial}{\partial w} - \hbar\frac{\partial}{\partial\hbar}.
    \end{align}
It is proved in \cite{MSS} that 
$P$ annihilates the partition function 
$Z^{(r)}(z,\hbar)$:
\begin{equation}
\label{eq:MSS}
PZ^{(r)}(z,\hbar)=
\left(
\hbar\frac{\partial}{\partial w}
+
e^{-\frac{r-1}{2}\hbar\frac{\partial}{\partial
    w}}e^{-rw} e^{\frac{r-1}{2}\hbar\frac{\partial}{\partial w}}e^{-r\hbar\frac{\partial}{\partial
    w}}
    \right)
    Z^{(r)}(z,\hbar)= 0.
\end{equation}
Since
$e^{\frac{r-1}{2}\hbar\frac{\partial}{\partial
    w}}$ is a \emph{shift} operator,
    with the multiplication 
    operator by a function
    $f(w)$ it satisfies the relation
$$
e^{\frac{r-1}{2}\hbar\frac{\partial}{\partial
    w}}\cdot f(w) = f(w+\frac{r-1}{2}\hbar) \cdot
    e^{\frac{r-1}{2}\hbar\frac{\partial}{\partial
    w}}.
    $$
Therefore, the operator $P$ can also be 
written as
\begin{displaymath}
  P= \hbar\frac{\partial}{\partial w} + e^{r(-w+\frac{r-1}{2}\hbar)}e^{-r\hbar\frac{\partial}{\partial
    w}}.
\end{displaymath}
Now the commutator relation
\begin{displaymath}
  [P,Q] = P
\end{displaymath}
is straightforward. 

The semi-classical limit calculations are the
same as those in \cite{MS}.
We have thus completed the proof 
of Theorem~\ref{thm:quantum curve}.
\end{proof}

\section{The Eynard-Orantin topological 
recursion}
\label{sect:EO}

In this section, we shall prove 
Theorem~\ref{thm:r-EO}.
For a mathematical definition of the Eynard-Orantin
theory, we refer to \cite{DMSS, MS}.

Because of the definition of the 
differentials 
$$
W_{g,n}^{(r)} = d_1\cdots d_n F_{g,n}^{(r)},
$$
we expect that the exterior differentiation 
of (\ref{eq:diffrecursion}) should give the 
integral recursion (\ref{eq:r-EO}). 
This naive idea does not work because of the
specific reference to the local Galois conjugation
$s_j$ appearing in the integral recursion. 
The PDE (\ref{eq:diffrecursion})
does not care about the $x$-projection of the
spectral curve, while (\ref{eq:r-EO}) heavily uses
the local ramification structure of the spectral 
curve as a covering of the $x$-coordinate line.
The integration kernel (\ref{eq:kernel}) shows
that the residue calculation on the right-hand side
of (\ref{eq:r-EO}) is similar to the local Galois
averaging. Yet evaluation of the free energies 
at any Galois conjugate point is no longer a rational 
function, since $s_j(z)$ is a very complicated
holomorphic function in $z$. 

The strategy we adopt in this section is
to extract the \emph{principal part} of 
the local Galois average, and then take the 
terms of the result that are the pull-back of
a function in the $x$-coordinate. On the
stable range $2g-2+n>0$, the  
free energies 
are indeed functions in the $x_i$-variables, so 
the last step makes sense. And by taking the
principal part of the Galois average, we maintain
the finiteness (polynomial-like)
structure of $W_{g,n}^{(r)}$
that represents the picewise polynomiality of 
the orbifold Hurwitz number 
$H_{g,n}^{(r)}(\vec{\mu})$.

Thus the simple residue operation of
the right-hand side of (\ref{eq:r-EO})
amounts to the combination of
the algebraic operations listed in 
Subsection~7.5 and the projection to the 
principal part described in 
Definition~\ref{def:proper-rational}.

\subsection{The spectral curve and the $x$-projection}
For the  convenience of calculations
we shall use the scaled coordinate 
$\eta = \sqrt[r]{r} z$ from now on. This 
change has no
significance, but some formulas and statements
become less cumbersome in the $\eta$-coordinate.

The $r$-Lambert curve (\ref{eq:r-Lambert}) is now
given by
\begin{equation*}
  x =  \frac{1}{\sqrt[r]{r}}\; \eta \; e^{-\eta^r/r},
\end{equation*}
and the $x$-projection has $r$ simple ramification
points at the $r$-roots of unity $1-\eta^r=0$. 
We denote these ramification points by
$$
\{\a_j\;|\;\a_j=e^{2(j-1)\pi i/r}, j=1,2,\cdots,r\}.
$$
Around each critical point $\a_j$, the 
$x$-projection is locally a double-sheeted 
covering.  
There is a neighborhood $U_j$ of $\a_j$
 such that when $\eta\in U_j$, there is another point
$\etatil$ satisfying $x(\etatil)=x(\eta)$. This 
correspondence defines a local deck transformation 
(or local Galois conjugation)
$s_j(\eta) := \etatil$  on
$U_j$. Clearly, $s_j$ is an
involution: $s_j(s_j(\eta)) = \eta$.

\begin{lem}
\label{lem:deck-tf}
For each $j=1,\cdots,r$, the
  deck transformation $s_j(\eta)$ is
  a holomorphic function in $\eta$ defined on
   $U_j$. Moreover, the function
   form of  $s_j(\eta)$ in the variable $\eta$ does
   not depend on the index $j$. 
\end{lem}

\begin{proof}
Let us introduce notations
$\Dt_j :=1-s_j(\eta)^r$ and $\Dt := 1-\eta^r$.
  The equation $x(s_j(\eta)) = x(\eta)$
  then gives
  \begin{displaymath}
    \log(1-\Dt_j)+\Dt_j = \log(1-\Dt) +\Dt.
  \end{displaymath}
  We make
$U_j$ smaller so that it lies in  the
  region $|\Dt|<1$. Then $\Dt_j$ has  a
   power series expansion 
  \begin{displaymath}
    \Dt_j=-\Dt -\frac23\Dt^2
    -\frac49\Dt^3-\frac{44}{135}\Dt^4-\frac{104}{405}\Dt^5-\frac{40}{189}\Dt^6
    -\frac{7648}{42525}\Dt^7-\frac{2848}{18225}\Dt^8 + O(\Dt^9),
  \end{displaymath}
  which convergences for $|\Dt|<1$. 
  Therefore $\Dt_j$ is a holomorphic function of
  $\eta$ defined on $U_j$ whose function
  form in $\eta$ does not 
  depend on $j$. Since
    \begin{displaymath}
    s_j(\eta) = \eta \exp\frac{\Dt-\Dt_j}{r},
  \end{displaymath}
  it is holomorphic in $\eta$ on
  $U_j$, and the function form
  does not depend on $j$, either. 
\end{proof}

\subsection{The free energies and the 
auxiliary functions in the $\eta$-coordinate}
By abuse of notation, we denote the auxiliary 
functions
of (\ref{eq:xi}) and (\ref{eq:xi-diff}) by the
same notation and consider them as functions in 
$\eta$. Thus we re-define 
\begin{equation}
\label{eq:xi-in-eta}
  \xi_{-1}^{r,k}(\eta) =
  \begin{cases}
   \frac{1}{r} \eta^r & k=0\\
    \frac1{kr^{k/r}}\eta^k & k>0,
  \end{cases}
  \qquad
   \xi_{m+1}^{r,k}(\eta) 
  = \frac{\eta}{1-\eta^r}
  \frac{d}{d\eta}
  \xi_m^{r,k}(\eta), \quad m\ge -1.
\end{equation}
\begin{rem}
  It is easy to see that $\xi_m^{r,k}(\eta)$ is 
  a proper rational function in $\eta$ for $m\ge0$, 
  whose
  denominator is a constant
  times $(1-\eta^r)^{2m+1}$. Thus
  $\xi_m^{r,k}(\eta)$ is meromorphic
  with poles only  at $\a_j$'s.
\end{rem}
A few examples of
 $\xi_m^{r,k}(\eta)$ are given in 
 Table~\ref{tab:xi}.
\begin{table}[ht]
  \centering
  \begin{tabular}{l|llll}
    & $k=0$ & $k=1$ & $k=2$ \\
    \hline
    & & & \\
    $m=-1$ & 
    $\frac{\eta^r}{r}$ & 
    $\frac{\eta}{r^{1/r}}$ & 
    $\frac{\eta^2}{2 r^{2/r}}$ \\
    &  &  & \\
    $m=0$ & 
    $\frac{\eta^r}{1-\eta^r}$ & 
    $\frac{\eta}{r^{1/r}(1-\eta^r)}$ & 
    $\frac{\eta^2}{r^{2/r}(1-\eta^r)}$\\
    & & &\\
    $m=1$ & 
    $\frac{r \eta^r}{\lf(1-\eta^r\rt)^3}$ & 
    $\frac{ \eta\,\lf[(r-1)\eta^r+1\rt]}{r^{1/r}\lf(1-\eta^r\rt)^3}$& 
    $\frac{\eta^2\,\lf[(r-2)
        \eta^r+2\rt]}{r^{2/r} \lf(1-\eta^r\rt)^3}$ \\ 
    & & &\\
    $m=2$ & 
    $\frac{r^2 \eta^r \lf(2 \eta^r+1\rt)}{\lf(1-\eta^r\rt)^5}$ & 
    $\frac{
      \eta\,\lf[(2r^2-3r+1) \eta^{2r}+(r^2+3r-2) \eta^r +1\rt]}{r^{1/r}\lf(1-\eta^r\rt)^5}$ & 
    $\frac{\eta^2\,\lf[2(r^2-3r+2)\eta^{2r}+(r^2+6r-8)\eta^r+4\rt]}{r^{2/r}
      \lf(1-\eta^r\rt)^5}$ \\ 
    & & &
  \end{tabular}
  \caption{$\xi_m^{r,k}(\eta)$ for $m=-1,\dots,2$ and $k=0,1,2$.}
  \label{tab:xi}
\end{table}

We denote the free energy 
$F_{g,n}^{(r)}(\eta_1,\dots,\eta_n)$ 
as
\begin{equation}\label{eq:Fgn-eta}
  F_{g,n}^{(r)}(\eta_1,\dots,\eta_n) = 
  r^{1-g}\,
  \sum_{\substack{|\vk|\equiv0\,(r)\\ |\vd|\le3g-3+n}}
  r^{|\vk|/r}\;
  \la \tau_{\vd}\;\Lam \ra^{(r),\vk}\;
  \xi_{\vd}^{r,\vk}(\eta_1,\dots,\eta_n),
\end{equation}
where $\vd:=(\ell_1,\dots,\ell_n)$ with $\ell_i\ge0$,
 $\vk:=(k_1,\dots,k_n)$ with $0\le
k_i<r$,
 $|\vk|:=\sum_{i=1}^n k_i$,
  $\Lam := \sum_{j\ge0}(-r)^j\lam_j$,
  and,
$\xi_{\vd}^{r,\vk}(\eta_1,\dots,\eta_n) :=\prod_{i=1}^n \xi_{\ell_i}^{r,k_i}(\eta_i)$.
The Hodge integrals are abbreviated as
   \begin{equation}
   \label{eq:Hodge}
  \la \tau_{\vd}\;\Lam \ra^{(r),\vk} := \la \tau_{\ell_1}\;\tau_{\ell_2}\dots\tau_{\ell_n} \;\Lam\ra^{(r),\vk} =
  \int_{\Mbar_{g,-\vk}(BG)}
  \prod_{i=1}^n \psi_i^{\ell_i}
  \sum_{j\ge 0}(-r)^j \lam_j.
\end{equation}
  In terms of the $\eta$-variables, the 
  unstable free energies 
  are given by
    \begin{align}
    \label{eq:F01-in-eta}
    F^{(r)}_{0,1}(\eta) &= \frac{1}{r^2}\eta^r-\frac{1}{2r^2}\eta^{2r},
    \\
    \label{eq:F02-in-eta}
    F^{(r)}_{0,2}(\eta_1,\eta_2) &=\log\frac{\eta_1-\eta_2}{x_1-x_2}
    -\frac{1}{r}(\log r + \eta_1^r+\eta_2^r).
  \end{align}
For $(g,n)$ in the stable range
 $2g-2+n>0$, 
  (\ref{eq:diffrecursion}) becomes
  \begin{multline}
    \label{eq:diffrecursion-in-eta}
    \lf( 2g-2+n+\frac{1}{r}\sum_{i=1}^n \eta_i \frac{\pp}{\pp \eta_i} \rt)
    F^{(r)}_{g,n}(\eta_1,\dots,\eta_n)
    \\
    = \half\sum_{i\ne j}\frac{\eta_i\eta_j}{\eta_i-\eta_j} 
    \lf[ \frac{1}{(1-\eta_i^r)^2}
    \frac{\pp}{\pp\eta_i} F^{(r)}_{g,n-1}
    \big(\eta_{[\hat{j}]}\big) 
    - \frac{1}{(1-\eta_j^r)^2}
    \frac{\pp}{\pp\eta_j} 
    F^{(r)}_{g,n-1}\big(\eta_{[\hat{i}]}\big) \rt]
    \\
    + \half \sum_{i=1}^n \frac{\eta_i^2}{(1-\eta_i^r)^2} 
    \lf.  \frac{\pp^2}{\pp u_1\pp u_2}
    F^{(r)}_{g-1,n+1}\big( u_1,u_2,\eta_{[\hat{i}]} \big)\right|_{u_1=u_2=\eta_i}
  \\
  + \half \sum_{i=1}^n \frac{\eta_i^2}{(1-\eta_i^r)^2} \sum_{\substack{ g_1+g_2=g\\I\sqcup J=[\hat{i}]}}
  ^{\rm{stable}} \lf( \frac{\pp}{\pp \eta_i} F^{(r)}_{g_1,|I|+1}(\eta_i,\eta_I) \rt) \lf(
  \frac{\pp}{\pp \eta_i} F^{(r)}_{g_2,|J|+1}(\eta_i,\eta_J) \rt).
\end{multline}

\subsection{The integration kernel for the 
Eynard-Orantin recursion}
Using (\ref{eq:F01-in-eta}) and (\ref{eq:F02-in-eta})
we find
$$
W_{0,1}^{(r)}(\eta) 
= \frac1r \eta^{r-1}(1-\eta^r)d\eta,
\qquad
W_{0,2}^{(r)}(\eta_1,\eta_2)
=\frac{d\eta_1\tensor d\eta_2}{(\eta_1-\eta_2)^2}
-\frac{dx_1\tensor dx_2}{(x_1-x_2)^2}.
$$
We recall  Lemma~\ref{lem:deck-tf}, which
states that the local Galois conjugation 
$s_j(\eta)$, considered as a function in $\eta$, 
does not depend on the index $j$.
Let us denote this function by
$\etatil = \etatil(\eta)$. As a consequence, 
the integration kernel (\ref{eq:kernel})
has an expression independent of $j$ as well, and 
is given by
\begin{equation}
\label{eq:kernel-in-eta}
K_j(\eta,\eta_1)
  =\frac r2\;
   \frac{\eta}{(\etatil^r-\eta^r)(1-\eta^r)}  
   \lf(\frac{1}{\eta-\eta_1}-\frac1{\etatil-\eta_1}\rt)
\cdot 
  d\eta_1\tensor\frac{1}{d\eta}.
\end{equation}

Here we note that $W_{0,2}^{(r)}$ in the 
integral recursion (\ref{eq:r-EO}) can be
replaced by Riemann's normalized fundamental 
differential of the second kind 
\begin{equation}
\label{eq:B}
B(\eta_1,\eta_2) = \frac{d\eta_1\tensor d\eta_2}
{(\eta_1-\eta_2)^2}.
\end{equation}
This is because the difference 
$$
W_{0,2}^{(r)}(\eta_1,\eta_2)-
B(\eta_1,\eta_2) = -\frac{dx_1\tensor dx_2}
{(x_1-x_2)^2}
$$
does not contribute to the residue calculation
of the second line of the right-hand side of
(\ref{eq:r-EO}).

\subsection{The local analytic properties of 
the auxiliary functions}

  Let us denote
  \begin{displaymath}
    \Da = 1-\eta^r  \quad \text{and} \quad \Db = 1-\etatil^r 
  \end{displaymath}
  for $\eta$ in each neighborhood $U_j$ of 
  the critical point $\a_j$.
When $\eta\to \a_i$, $\Da(\eta)$ and $\Db(\eta)$   converge to $0$. We therefore regard $\Da$ and
$\Db$  as small parameters, for example, 
$|\Da|<1$ and $|\Db|<1$, for $\eta\in U_j$.

To analyze the $r$-Lambert curve locally 
around its critical point, let us introduce
a local parameter $u$ around $\a_j$ by
\begin{equation}
\label{eq:u}
  e^{-u-1} = \eta^r e^{-\eta^r} = \etatil^r e^{-\etatil^r}.
\end{equation}
When $\eta$ is in any neighborhood $U_j$, 
we have
\begin{equation}
\label{eq:u-Da-Db}
  u = -\Da-\log(1-\Da) = -\Db-\log(1-\Db).
\end{equation}
Since the $u$-projection of the $r$-Lambert
curve around $\a_j$ is a double-sheeted 
covering, let us define
\begin{equation} \label{eq:v-u}
  \half v^2 = u.
\end{equation}
This \emph{Airy curve} equation describes
the local behavior of our spectral curve around $\a_j$.
We choose the branch of 
 $v$ at $\a_j$ so that we have an expansion
\begin{equation}\label{eq:v-in-Delta}
   v = \Da + \frac13 \Da^2 + \frac7{36}\Da^3 + \frac{73}{540}\Da^4+ O(\Da^5).
\end{equation}
Here again we can see that\
 $v$ is a holomorphic function in $\eta$ with the same expression at each neighborhood $U_j$
of $\a_j$, without any explicit dependence
on the index $j$. From the
definition and (\ref{eq:u-Da-Db}), we know 
that the other branch of the curve around $\a_j$  
is given by
\begin{displaymath}
  v(\etatil) = \Db+\frac13 \Db^2 + \frac7{36}\Db^3 + \frac{73}{540}\Db^4+ O(\Db^5).
\end{displaymath}
Of course in terms of the $v$-coordinate the 
local Galois conjugate is given simply by
\begin{equation}\label{eq:v and -v}
  v(\etatil) = -v(\eta),
\end{equation}
while $u$ is symmetric under the involution $\eta\mapsto\etatil$.

At each $\a_j$,  we can  express $\Da$ and $\Db$ as inverse series in $v$ for sufficiently small $v$.
\begin{equation}
\label{eq:Da,Db}
\begin{aligned}
\Da &=\psi(v):= v - \frac13 v^2 + \frac1{36} v^3 +\frac1{270} v^4 + O(v^5),
\\
 \Db &= \psi(-v) = -v- \frac13 v^2 - \frac1{36} v^3 +\frac1{270} v^4 + O(v^5).
  \end{aligned}
\end{equation}
To make the equations shorter,
we use the following functions
  \begin{align}
    \label{eq:Y}
    Y^k(\eta) &= \frac{\eta^k+\etatil^k}{2},\\
    \label{eq:E-func}
    E^{r,k}(\eta) &= \frac{\eta\,\frac{d}{d\eta}Y^k(\eta)}{(1-\eta^r)\, Y^k(\eta)},\\
    \label{eq:phi}
    \phi_n^{r,k}(\eta) &=
    \dfrac{\xi_n^{r,k}(\eta)-\xi_n^{r,k}(\etatil)}{2\,Y^k(\eta)},\\
    \label{eq:h-func}
    h_n^{r,k}(\eta) & = \dfrac{\xi_n^{r,k}(\eta)+\xi_n^{r,k}(\etatil)}{2\,Y^k(\eta)}.
  \end{align}
  Here $\etatil$ is understood as $s_j(\eta)$ for any $j$.

\begin{rem}
  Again thanks to Lemma~\ref{lem:deck-tf}, the 
  above expressions are well defined and 
 independent of which
  neighborhood $U_j$ of the
  critical points the $\eta$-variable lies.
\end{rem}

\begin{prop}\label{prop:analytic-properties-phi-h}
The functions
  $Y^k(\eta)$, $E^{r,k}(\eta)$ and $h_n^{r,k}(\eta)$ are symmetric  under the involution
  $\eta\mapsto\etatil$, while $\phi_n^{r,k}(\eta)$ is anti-symmetric. In terms of the local 
  parameter $v$, we have
  \begin{enumerate}
  \item [(i)] $E^{r,k}$ is an even holomorphic function in $v$, or a holomorphic function in $u$;
  \item [(ii)] $\phi_{-1}^{r,k}$ is an odd holomorphic function in $v$. For $n\ge0$, $\phi_n^{r,k}$
    is an odd meromorphic function in $v$, which has at most $(2n+1)$-th order pole at $v=0$, and
    no other poles near $v=0$;
  \item [(iii)] For $n\ge-1$, $h_n^{r,k}$ is an even holomorphic function in $v$, or a holomorphic
    function in $u$.
  \end{enumerate}
\end{prop}

\begin{proof}
  (i). By definition
  it is clear that $Y^k(\eta)$ and $h_n^{r,k}(\eta)$ are symmetric,  and $\phi_n^{r,k}$ is
   anti-symmetric,  under the involution.
  Note that $E^{r,k}(\eta)$ has a local expression
  \begin{align*}
    E^{r,k}(\eta) = k\dfrac{\psi(-v) e^{\frac{k}{r}\psi(-v)}+\psi(v)
      e^{\frac{k}{r}\psi(v)}}{\psi(v)\psi(-v)\lf(e^{\frac{k}{r}\psi(v)}+e^{\frac{k}{r}\psi(-v)}\rt)}
    =k \dfrac{(\frac{k}{r}-\frac{1}{3})+O(v)}{1+O(v)}.
  \end{align*}
  From the first equality we know that $E^{r,k}$ is  a function in $v$, and symmetric  under the involution. The second equality  is due to the expansion
  (\ref{eq:Da,Db}) of $\psi(v)$, which indicates that$E^{r,k}(v)$ is holomorphic near $v=0$. Thus near
  $v=0$, $E^{r,k}$  expands into a
   power series containing only even powers of $v$. Hence
  $E^{r,k}$ is a power series in $u=\half v^2$. 
    
We now prove (ii) and (iii) by induction. 
  (ii). For $n=-1$, when $\eta\in U_j$, we have
  \begin{equation}
  \begin{aligned}\label{eq:phi-1-r-k}
    \phi_{-1}^{r,k} 
    &= 
    \frac1{k r^{k/r}} 
    \dfrac{\eta^k-\eta^k e^{\frac{k}{r}(\Da-\Db)}}{\eta^k+\eta^k\,e^{\frac{k}{r}(\Da-\Db)}}
    \\
    &= 
    \frac1{k r^{k/r}} 
    \tanh
    \lf[\frac{k\,\lf(\psi(-v)-\psi(v)\rt)}{2r}\rt]
    = -\frac1{r^{1+k/r}}v+O(v^3) \quad
    {\text{for}}\quad
    k>0,
    \\
    \phi_{-1}^{r,0}
    &=
    \frac{\psi(-v)-\psi(v)}{2r} = -\frac{v}{r}+O(v^3).
  \end{aligned}
  \end{equation}
These are odd holomorphic functions near $v=0$. By induction, for $n\ge-1$, if
  $\phi_n^{r,k}$ is an odd function in $v$, then  we have:
  \begin{multline*}
    2\,Y^k(\eta)\,\phi_{n+1}^{r,k}(\eta) 
    = \xi_{n+1}^{r,k}(\eta)-\xi_{n+1}^{r,k}(\etatil) 
    = \frac{\eta}{1-\eta^r}\frac{d}{d\eta}\lf(\xi_{n}^{r,k}(\eta)-\xi_{n}^{r,k}(\etatil)\rt)
    \\
    =\frac{\eta}{1-\eta^r}\frac{d}{d\eta}\lf(2\,Y^k(\eta)\,\phi_{n}^{r,k}(v)\rt)
    =2\,Y^k(\eta)\,\lf(E^{r,k}(u)-\frac{r}{v}\frac{d}{dv}\rt)\phi_n^{r,k}(v).
  \end{multline*}
  Here we used (\ref{eq:E-func})
  and 
    \begin{equation}
    \label{eq:vector-field-identities}
    \frac{\eta}{1-\eta^r}\frac{d}{d\eta} = \frac{\etatil}{1-\etatil^r}\frac{d}{d\etatil} =
    -\frac{r}{v}\frac{d}{dv} = -r\frac{d}{du}.
  \end{equation}
    Therefore we obtain a recursion formula
  \begin{equation}\label{eq:recursionPhi}
    \phi_{n+1}^{r,k} = \lf(E^{r,k}(u)-\frac{r}{v}\frac{d}{dv}\rt)\phi_n^{r,k}(v),
  \end{equation}
  which proves that $\phi_{n}^{r,k}$ for $n\ge0$ are odd meromorphic functions of
  $v$, with poles of order at most $2n+1$ 
  at $v=0$ and no other poles near $v=0$. 
  
  (iii). Note that
  $h_{-1}^{r,k}=\frac{1}{k r^{k/r}}$ is  an even holomorphic function in
  $v$. For $n\ge-1$, suppose that
   $h_n^{r,k}$ is a function in $u$. Then the recursion
  \begin{equation}
    \label{eq:recursion-h}
    h_{n+1}^{r,k} = \lf(E^{r,k}(u)-r\frac{d}{du}\rt)\,h_n^{r,k}(u)
  \end{equation}
  shows that $h_{n+1}^{r,k}$ is again an
  even holomorphic function in $v$.
  This completes the proof.
\end{proof}

\begin{rem}\label{rem:xi-etabar}
 Around each critical point 
 $\a_j$, we have
  \begin{equation}\label{eq:xi-decompose-in-phi-h}
    \xi_n^{r,k}(\eta) = Y^k(\eta)\lf(\phi_n^{r,k}(v)+h_n^{r,k}(u)\rt),\qquad
    \xi_n^{r,k}(\etatil) = Y^k(\eta)\lf(-\phi_n^{r,k}(v)+h_n^{r,k}(u)\rt).
  \end{equation}
  For $k=0$, we have
  \begin{equation}
    \label{eq:phi-1}
    \phi_{-1}^{r,0}(v) = \frac{\eta^r-\etatil^r}{2r} = \frac{\Db-\Da}{2r},
  \end{equation}
  \begin{equation}
    \label{eq:Da-in-phi}
    \Da = 1-r\phi_{-1}^{r,0}(v) - r h_{-1}^{r,0}(u),\quad
    \Db = 1+r\phi_{-1}^{r,0}(v) - r h_{-1}^{r,0}(u),
  \end{equation}
  and
  \begin{displaymath}
    \phi_{n+1}^{r,0}(v) = -\dfrac{r}{v}\frac{d}{dv}\phi_n^{r,0}(v), \qquad h_{n+1}^{r,0}(u) =
    -r\frac{ d}{du} h_{n}^{r,0}(u).
  \end{displaymath}
\end{rem}

\subsection{The local Galois averaging}

The shape of the Eynard-Orantin integral
recursion (\ref{eq:r-EO}), together with the integration
kernel given by
 (\ref{eq:kernel-in-eta}) 
 and the local Galois conjugation 
 of (\ref{eq:v and -v}), suggests that the
 residue evaluation of the right-hand-side
 of (\ref{eq:r-EO}) is equivalent to the 
 local Galois averaging with respect to the
 single variable $z$. 
 Since we already have a topological recursion
 in the form of the partial differential equation
 (\ref{eq:diffrecursion}), it is natural to 
 expect that the local Galois averaging of
 (\ref{eq:diffrecursion}) should produce 
 (\ref{eq:r-EO}). 
 In this subsection we apply 
 the following three algebraic
 operations to the differential equation
 (\ref{eq:diffrecursion}). 
 \begin{enumerate}
 \item Local Galois averaging with respect to 
 the first variable $\eta_1$. This means that for
 a meromorphic function $f(\eta_1)$ defined
 on $U_j$, we apply 
 \begin{displaymath}
  \pi_*: f(\eta_1)\,\mapsto\, \frac{f(\eta_1) + f(\etatil_1)}{2}.
\end{displaymath}
\item Extract the part of the function that is 
symmetric with respect to the local 
Galois conjugation (\ref{eq:v and -v}).
In this process we use   
Proposition~\ref{prop:analytic-properties-phi-h}
and (\ref{eq:Da-in-phi}).
\item To make the matching of our 
formula with the integral
recursion manifest, we then multiply by the 
factor 
$$
\frac{v_1\,dv_1}{r\phi_{-1}^{r,0}(v_1)}.
$$
Because of (\ref{eq:Da,Db}) and (\ref{eq:phi-1}),  $\phi_{-1}^{r,0}(v_1) = O(v_1)$ near $v_1=0$. 
Therefore,
the multiplication factor $\frac{v_1}{r\phi_{-1}^{r,0}(v_1)}$ is a holomorphic function around each 
critical point $\a_j$.
 \end{enumerate}

Our starting point is the following
Laplace transform formula
(\ref{eq:diffrecursion-in-eta})
of the cut-and-join equation (\ref{eq:CAJ}),
written in terms of the Hodge integrals
(\ref{eq:Hodge})
and the auxiliary functions (\ref{eq:xi-in-eta})
incorporating the ELSV-type
formula (\ref{eq:ELSV}) of \cite{JPT}.

\begin{multline}
\label{eq:LF-in-xi}
  \sum_{\substack{|\vk|\equiv 0\\ |\vd|\le 3g-3+n}}
  r^{\frac{|\vk|}{r}}\,
  \la\ta_{\vd}\,\Lam\ra^{r,\vk}\,
  \lf[
  (2g-2+n)\,\xi_{\vd}^{r,\vk}(\eta_{[n]})
  +\frac1r\sum_{i=1}^n
  (1-\eta_i^r)\,
  \xi_{\ell_i+1}^{r,k_i}(\eta_i)\,
  \xi_{\dWithoutI}^{r,\kWithoutI}(\etaWithoutI)
  \rt]
  \\
  =
  \sum_{i<j}
  \sum_{\substack{a+\lf|\kWithoutIJ\rt|\equiv0 \\ m+\lf|\dWithoutIJ\rt|\le 3g-4+n}}
  r^{\frac{{a+\lf|\kWithoutIJ\rt|}}{r}}\,
  \la\ta_m\,\ta_{\dWithoutIJ}\,\Lam\ra^{(r),(a,\kWithoutIJ)}\,
  \frac{1}{\eta_i-\eta_j}
  \\
  \times
  \lf[
  \frac{\eta_j\,\xi_{m+1}^{r,a}(\eta_i)}{1-\eta_i^r}\,
  -\frac{\eta_i\,\xi_{m+1}^{r,a}(\eta_j)}{1-\eta_j^r}\,
  \rt]
  \xi_{\dWithoutIJ}^{r,\kWithoutIJ}(\etaWithoutIJ)
  \\
  +
  \frac{r}{2}
  \sum_{i=1}^n
  \sum_{\substack{a+b+\lf|\kWithoutI\rt|\equiv 0\\m+\ell+\lf|\dWithoutI\rt|\le 3g-5+n}}
  r^{\frac{a+b+\lf|\kWithoutI\rt|}{r}}\,
  \la\ta_m\,\ta_\ell\,\ta_{\dWithoutI}\,\Lam\ra^{(r),(a,b,\kWithoutI)}\,
  \xi_{m+1}^{r,a}(\eta_i)\,
  \xi_{\ell+1}^{r,b}(\eta_i)\,
  \xi_{\dWithoutI}^{r,\kWithoutI}(\etaWithoutI)
  \\
  +
  \frac{r}{2}
  \sum_{i=1}^n
  \sum_{\substack{g_1+g_2=g\\I\sqcup J=[\hat{i}]}}^{\text{stable}}
  \sum_{\substack{
      a+|k_I|\equiv0\\
      b+|k_J|\equiv0\\
      m+|\ell_I|\le 3g_1-2+|I|\\
      \ell+|\ell_J|\le 3g_2-2+|J|}}
  r^{\frac{a+b+\lf|\kWithoutI\rt|}{r}}
  \la\ta_m\,\ta_{\ell_I}\,\Lam\ra^{r,(a,k_I)}
  \la\ta_\ell\,\ta_{\ell_J}\,\Lam\ra^{r,(b,k_J)}
  \\
  \times
  \xi_{m+1}^{r,a}(\eta_i)\,
  \xi_{\ell+1}^{r,b}(\eta_i)\,
  \xi_{\dWithoutI}^{r,\kWithoutI}(\etaWithoutI).
\end{multline}
Here the bound of the summation indices are
 $0\le a,b<r$ and $m,\ell\ge0$.
 Let us now apply the three algebraic operations
 listed above to (\ref{eq:LF-in-xi}).

The left-hand-side of 
(\ref{eq:LF-in-xi}) produces 
\begin{equation}\label{eq:LF-changes-in-1st-line}
  \sum_{\substack{\lf|\vk\rt|\equiv0\\ \lf|\vd\rt|\le 3g-3+n}}
  r^{\frac{\lf|\vk\rt|}{r}}\,
  \la\tau_{\vd}\,\Lam\ra^{(r),\vk}~
  Y^{k_1}(\eta_1)
  \lf[  
  -\frac{v_1}{r}\,\phi_{\ell_1+1}^{r,k_1}(v_1)\,dv_1
  +\mathcal{H}(v_1)\,dv_1
  \rt]
  \xi_{\dWithoutOne}^{r,\kWithoutOne}(\etaWithoutOne),
\end{equation}
where
$\mathcal{H}(v_1)$ is a holomorphic function in $v_1$ near $v_1=0$ that comes from the
third operation.
We calculate, using,
(\ref{eq:recursionPhi}) and (\ref{eq:xi-decompose-in-phi-h}):
\begin{multline*}
  -\frac{v_1}{r}\,
  Y^{k_1}(\eta_1)\,
  \phi_{\ell_1+1}^{r,k_1}(v_1)\,dv_1
  =
  -\frac{v_1}{r}\,
  Y^{k_1}(\eta_1)\,
  \lf[E^{r,k_1}(u_1) -\frac{r}{v_1} \frac{d}{dv_1}\rt]\,
  \phi_{\ell_1}^{r,k_1}(v_1)\,dv_1
  \\
  = \,
  Y^{k_1}(\eta_1)\,
  d\phi_{\ell_1}^{r,k_1}(v_1)
  -\frac{v_1}{r} 
  Y^{k_1}(\eta_1)\,
  E^{r,k_1}(u_1)\,
  \phi_{\ell_1}^{r,k_1}(v_1)\,dv_1 
  \\
  = \,
  d\xi_{\ell_1}^{r,k_1}(\eta_1) 
  - d\lf(\,Y^{k_1}(\eta_1) 
  h_{\ell_1}^{r,k_1}(u_1)\,\rt)
  \\
  -\lf(\,dY^{k_1}(\eta_1)\,\rt) 
  \phi_{\ell_1}^{r,k_1}(v_1)
  -\frac{1}{r} 
  Y^{k_1}(\eta_1)\,
  E^{r,k_1}(u_1)\,
  \phi_{\ell_1}^{r,k_1}(v_1)\;v_1\,dv_1.
\end{multline*}
The last two terms of the above
formula cancel due to (\ref{eq:E-func}) and 
\begin{equation}
  \label{eq:diff-relations}
  du = vdv = r\frac{\eta^r-1}{\eta}d\eta = r\frac{\etatil^r-1}{\etatil}d\etatil.
\end{equation}
Notice that $Y^{k_1}(\eta_1)$, $\mathcal{H}(v_1)$, and $h_{\ell_1}^{r,k_1}(u_1)$ are holomorphic
functions in $\eta_1\in U$, where $U$ is the union $\cup_{j=1}^r U_j$. Thus the left-hand side
of (\ref{eq:LF-in-xi}) simply takes the form
\begin{equation}
  \label{eq:lhs-final}
  \sum_{\substack{\lf|\vk\rt|\equiv0\\ \lf|\vd\rt|\le3g-3+n}}
  r^{\frac{\lf|\vk\rt|}{r}}\,
  \la \tau_{\vd}\,\Lam \ra^{(r),\vk}\,
  \lf[
  d\xi_{\ell_1}^{r,k_1}(\eta_1)
  +\mathcal{H}_1(\eta_1)\,d\eta_1
  \rt]\,
  \xi_{\dWithoutOne}^{r,\kWithoutOne}(\etaWithoutOne),
\end{equation}
with a holomorphic function 
$\mathcal{H}_1$ in $\eta_1$.

Again appealing to (\ref{eq:diff-relations}), 
we calculate the result of the three operations
on the first term of
the right-hand side of (\ref{eq:LF-in-xi}) as
\begin{multline}\label{eq:LF-changes-in-2nd-line}
    \half\,
    \sum_{1<j}
    \sum_{\substack{
        a+\lf|\kWithoutOneJ\rt|\equiv0\\
        m+\lf|\dWithoutOneJ\rt|\le3g-4+n}}
    r^{\frac{a+\lf|\kWithoutOneJ\rt|}{r}}
    \la
    \tau_m\,
    \tau_{\dWithoutOneJ}\,
    \Lam
    \ra^{(r),(a,\kWithoutOneJ)}
    \\
    \times
    \lf[
    \frac{-\xi_{m+1}^{r,a}(\eta_1)\,d\,\eta_1}{(\eta_1-\eta_j)\,\phi_{-1}^{r,0}(v_1)}
    +
    \frac{-\xi_{m+1}^{r,a}(\etatil_1)\,d\,\etatil_1}{(\etatil_1-\eta_j)\,\phi_{-1}^{r,0}(v_1)}
    +\Omega_{1,m+1}^{r,a}(\eta_1) 
    +\Omega(\eta_1,\eta_j)\,
    \frac{\xi_{m+1}^{r,a}(\eta_j)}{1-\eta_j^r}
    \rt]
    \\
    \times
    \xi_{\ell_{[\hat 1,\hat j]}}^{r,k_{[\hat 1,\hat
        j]}}
    (\eta_{[\hat 1,\hat j]})
    +\mathcal{H}_2(\eta_1)\,d\eta_1\,\mathcal{F}(\etaWithoutOne).
\end{multline}
Here 
$\mathcal{H}_2(\eta_1)\in\cO(U)$.
$\mathcal{F}(\etaWithoutOne)$ is a function in
$\eta_2,\dots,\eta_n$. $\Omega_{1,m+1}^{r,a}(\eta_1)$ is a meromorphic differential in $\eta_1$
defined by
\begin{displaymath}
  \Omega_{1,m+1}^{r,a}(\eta_1) 
  := \frac{\xi_{m+1}^{r,a}(\eta_1)\,d\,\eta_1}{\eta_1\,\phi_{-1}^{r,0}(v_1)}
  + \frac{\xi_{m+1}^{r,a}(\etatil_1)\,d\,\etatil_1}{\etatil_1\,\phi_{-1}^{r,0}(v_1)},
\end{displaymath}
and $\Omega(\eta_1,\eta_j)$ is a holomorphic differential in $\eta_1$ 
\begin{displaymath}
  \Omega(\eta_1,\eta_j) = 
  \lf[
  \frac{(1-\eta_1^r)\,d\,\eta_1}{(\eta_1-\eta_j)\,\phi_{-1}^{r,0}(v_1)}
  +\frac{(1-\etatil_1^r)\,d\,\etatil_1}{(\etatil_1-\eta_j)\,\phi_{-1}^{r,0}(v_1)}
  \rt]\,
  \frac{\xi_{m+1}^{r,a}(\eta_j)}{1-\eta_j^r}.
\end{displaymath}
Without loss of generality,
we can assume that  $\eta_j\not\in U$.
Then $\Omega(\eta_1,\eta_j)$ is a holomorphic differential with respect to
 $\eta_1\in U$, which follows from the 
local behavior of 
$\frac{(1-\eta_1^r)}{\phi_{-1}^{r,0}(v_1)}$
and
$\frac{(1-\etatil_1^r)}{\phi_{-1}^{r,0}(v_1)}$
coming from 
 (\ref{eq:v-in-Delta}) and
(\ref{eq:phi-1}).

The operation on 
the second and the third terms of the right-hand side
of  (\ref{eq:LF-in-xi}) produces
\begin{multline}
  \label{eq:LF-changes-in-3rd-line}
    \sum_{\substack{a+b+\lf|\kWithoutOne\rt|\equiv0\\ m+\ell+\lf|\dWithoutOne\rt|\le 3g-5+n}}
    r^{\frac{a+b+\lf|\kWithoutOne\rt|}{r}}\,
    \la\tau_m\,\tau_\ell\,\tau_{\dWithoutOne}\,\Lam
    \ra^{(r),(a,b,\kWithoutOne)} 
    \\
    \times
    \lf[
    \frac{
      Y^a(\eta_1)\,
      Y^b(\eta_1)\,    
      \phi_{m+1}^{r,a}(v_1)\,
      \phi_{\ell+1}^{r,b}(v_1)\,
      v_1\,d\,v_1}
    {2\phi_{-1}^{r,0}(v_1)}\rt]
    \xi_{\dWithoutOne}^{r,\kWithoutOne}(\etaWithoutOne)
    +\mathcal{H}_3(\eta_1)\,d\,\eta_1\,\mathcal{F}_1(\etaWithoutOne)
    \\
    +
    \sum_{\substack{g_1+g_2=g\\I\sqcup J=[\hat1]}}^{\text{stable}}
    \sum_{\substack{
        a+\lf|k_I\rt|\equiv0\\
        b+\lf|k_J\rt|\equiv0\\
        m+\lf|\ell_I\rt|\le 3g_1-2+|I|\\
        \ell+\lf|\ell_J\rt|\le 3g_2-2+|J|}}
    r^{\frac{a+b+|\kWithoutOne|}{r}}\,
    \la\tau_m\,\tau_{\ell_I}\,\Lam \ra^{(r),(a,k_I)}\,
    \la\tau_\ell\,\tau_{\ell_J}\,\Lam\ra^{(r),(a,k_J)}\\
    \times
    \lf[
    \frac{
      Y^a(\eta_1)\,
      Y^{b}(\eta_1)\,
      \phi_{m+1}^{r,a}(v_1)\,
      \phi_{\ell+1}^{r,b}(v_1)\,
      v_1\,d v_1}
    {2\phi_{-1}^{r,0}(v_1)}
    \rt]
    \xi_{\dWithoutOne}^{r,\kWithoutOne}(\etaWithoutOne)
    +
    \mathcal{H}_4(\eta_1)\,d\eta_1\,
    \mathcal{F}_2(\etaWithoutOne).
\end{multline}
Here $\mathcal{H}_3(\eta_1)$ and 
$\mathcal{H}_4(\eta_1)$ are in $\cO(U)$,
and 
$\mathcal{F}_1(\etaWithoutOne)$ 
and 
$\mathcal{F}_2(\etaWithoutOne)$
are  functions in $\eta_2,\dots,\eta_n$.

\subsection{The residue calculation}

Recall that the central idea of \cite{EMS}
to prove the Hurwitz number conjecture
of \cite{BM} 
is to relate the \emph{principal part}
of the free energies
with the residue calculation of the
Eynard-Orantin integral recursion formula.
Since the free energies
in our case have $r$ distinct poles, we need a more
general notion of the principal part for a 
meromorphic function with many poles
(see for example, \cite{Vostry}).
In this subsection we derive the key formula
(\ref{eq:residue formula}) 
for the residue calculations we need.

\begin{Def}
Let us denote by $U=\cup_{j=1}^r U_j$
the union of the local neighborhood of the 
critical point $\a_j$ for all $j$. We define
an $\cO(U)$-module $\cM_U$ by
$$
\cM_U = \left\{
m(\eta)\;\left|\;m(\eta) = \frac{h(\eta)}{(1-\eta^r)^k},
k\ge 0, h(\eta)\in \cO(U).
\right.\right\}
$$
\end{Def}

Following \cite{Vostry}, we define

\begin{Def}
Let $p(\eta)\in \bC[\eta]$ be a non-constant
polynomial, and $U\subset \bC$ an open subset. 
   Two functions $f,g\in \cO(U)$ are 
  said to be congruent modulo $p$ (denoted by 
  $f\equiv g\mod p$)
   if there is 
  $q(\eta)\in \cO(U)$ such that
  \begin{displaymath}
    f = g + p\,q.
  \end{displaymath}
\end{Def}

\begin{prop}[\cite{Vostry}]
Under the same condition as above, suppose that
$p(\eta)$ has all its zeros in $U$.  
  Then for every 
   holomorphic function $f(\eta)\in\cO(U)$,
     there is a unique polynomial $r(\eta)$
     such that
     $$
     f \equiv r\mod p\qquad{\text{and}}\qquad
     \deg r(\eta) < \deg p(\eta).
     $$
     \end{prop}
     We denote this unique remainder polynomial by
\begin{equation}
\label{eq:remainder}
  r(\eta) = \lfloor f(\eta)\rfloor_p .
\end{equation}
This defines a natural $\bC$-algebra
homomorphism
$$
\cO(U)\owns f\longmapsto \lfloor f(\eta)\rfloor_p
\in \bC[\eta]\big/(p),
$$
which is called the \emph{reduction} of
  $f$ modulo $p$.

\begin{Def}\label{def:proper-rational}
  Let $m(\eta) = h(\eta)/p(\eta)$
   be a meromorphic function in $\cM_U$,
    where
  $h(\eta)\in \cO(U)$ and $p(\eta)=(1-\eta^r)^k$, $k\ge0$.   We define the following symbol
  \begin{equation}
    \label{eq:proper-rational}
    \lcb m(\eta) \rcb_\eta := 
    \frac{\lfloor h(\eta)\rfloor_p}{p(\eta)}
    \in\bC(\eta).
  \end{equation}
  Thus we have a linear map, 
  which we simply call the projection to the 
  \emph{principal part} 
  $$
  \{\;\cdot\;\}_\eta:\cM_U\lrar \bC(\eta).
  $$
 The principal part $\lcb m(\eta) \rcb_\eta$
 of  a meromorphic function $m(\eta)$ is the ``proper rational function
  part"  of $m=h/p$.  If $k=0$, then we define
  the principal part to be $0$.
\end{Def}

\begin{rem} \label{rem:properties-proper-rational}
  From the definition
  it is obvious  that
  for every $m(\eta)\in\cM_U$, we
  have
      \begin{displaymath}
    m(\eta)-\{m(\eta)\}_\eta \in \cO(U).
       \end{displaymath}
       Thus $\{m(\eta)\}_\eta$ behaves much like
       the principal part of a meromorphic function
       at a pole. 
The image $\{m(\eta)\}_\eta
\in\bC(\eta)$ is always globally defined
    on $\bP^1$, even though $m(\eta)$ is  defined locally on $U$, and $\{m(\eta)\}_\eta$ has poles only
    at $\a_j$.
 \end{rem}
 
 The following lemma plays the key role
 in connecting the residue of the 
 Eynard-Orantin recursion formula 
 and taking the principal part.

\begin{lem}\label{lem:residue-calculation}
  For any  element $m(\zeta)\in\cM_U$ 
  and $\eta\in\bC$ such that $\eta\neq\a_j$,
  $j=1,\dots,r$, we have
  \begin{equation}
  \label{eq:residue formula}
    \sum_{j=1}^r \Res_{\zeta=\a_j}
    \frac{m(\zeta)}{\zeta-\eta} = -\lcb m(\eta) \rcb_\eta.
  \end{equation}
\end{lem}

\begin{proof}
Let $\gamma_j$ be a
  small loop in $U_j$ centered at $\a_j$, 
   $\gamma_\eta$ a small loop around $\eta$, and
  $\Gamma_R$ a 
  large circle enclosing all $\a_j$ and $\eta$ with radius $R \gg 1$ (see Figure \ref{fig:path}).
  Then
  \begin{align*}
    &\sum_{j=1}^r \Res_{\zeta=\a_j}\frac{m(\zeta)}{\zeta-\eta} = 
    \sum_j\frac{1}{2\pi i}\oint_{\gamma_j}\frac{m(\zeta)}{\zeta-\eta}\,d\zeta
     \\
    =&\sum_j\frac{1}{2\pi i}\oint_{\gamma_j}\frac{\lcb m(\zeta)\rcb_\zeta}{\zeta-\eta}\,d\zeta + 
    \sum_j\frac{1}{2\pi i}\oint_{\gamma_j}\frac{m(\zeta)-\lcb m(\zeta)\rcb_\zeta}{\zeta-\eta}\,d\zeta\\
    =&\sum_j\Res_{\zeta=\a_j}\frac{\lcb m(\zeta)\rcb_\zeta}{\zeta-\eta},
  \end{align*}
  because $\frac{m(\zeta)-\lcb m(\zeta)\rcb_\zeta}{\zeta-\eta}$ does not have any pole in 
  any of the $U_j$'s.
Noting that
 $\frac{\lcb m(\zeta)\rcb_\zeta}{\zeta-\eta}$
 is a rational function with poles only at $\eta$ and
  $\a_j$, $j=1,\dots,r$,
 we calculate 
  \begin{align*}
    0 =& \lim_{R\rar\infty}\frac1{2\pi i}\oint_{\Gamma_R} \frac{\lcb m(\zeta)\rcb_\zeta}{\zeta-\eta}\,d\zeta \\
    =&\sum_j\Res_{\zeta=\a_j}\frac{\lcb m(\zeta)\rcb_\zeta}{\zeta-\eta} +
    \Res_{\zeta=\eta}\frac{\lcb m(\zeta)\rcb_\zeta}{\zeta-\eta} \\
    =&\sum_j\Res_{\zeta=\a_j}\frac{\lcb m(\zeta)\rcb_\zeta}{\zeta-\eta} +
    \lcb m(\eta)\rcb_\eta.
  \end{align*}
This completes the proof of 
(\ref{eq:residue formula}).
\end{proof}

\begin{figure}[htb]
\centerline{\epsfig{file=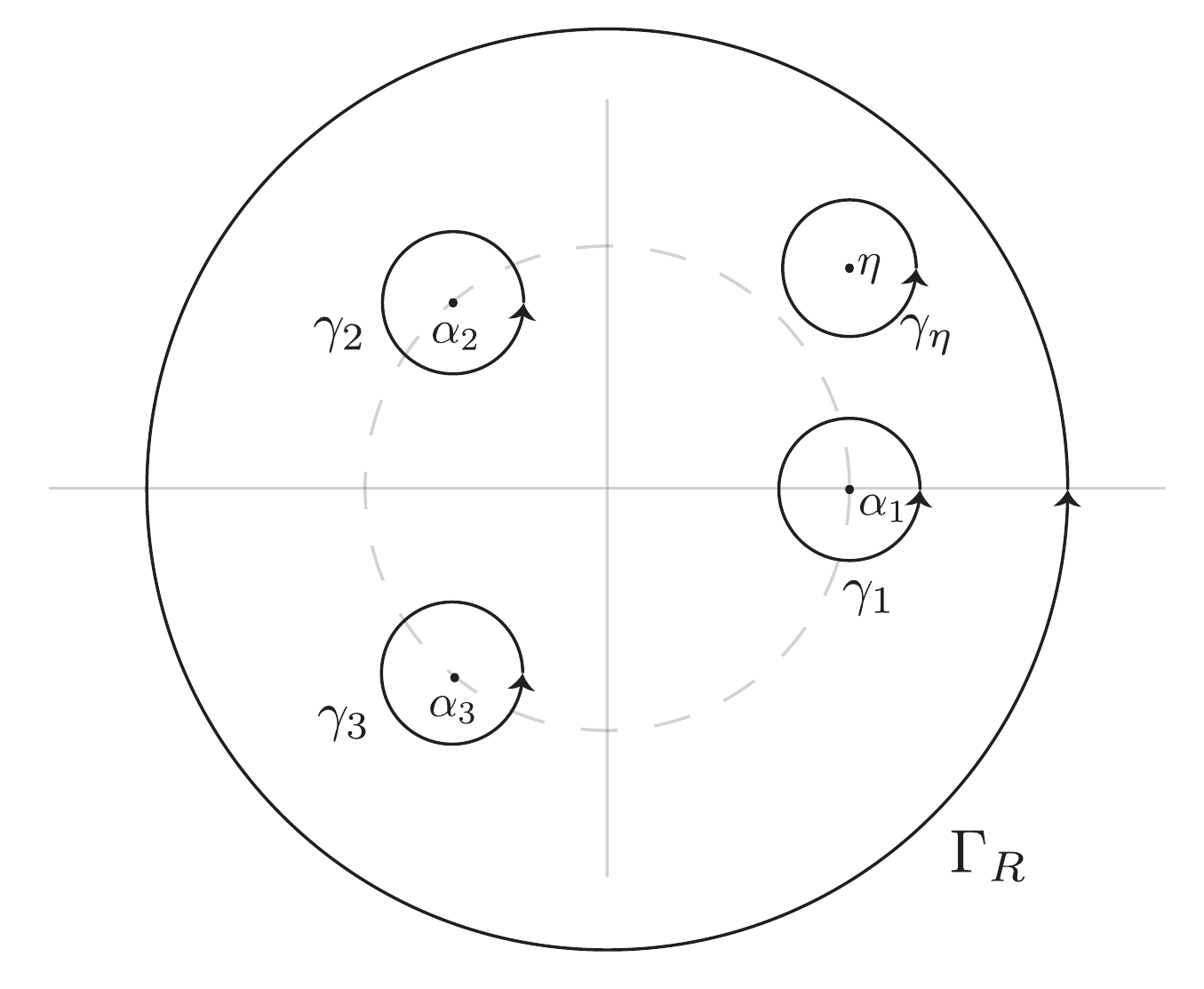, 
width=3in}}
\caption{Integration contours.}
    \label{fig:path}
\end{figure}

\subsection{Proof of Theorem~\ref{thm:r-EO}}

We are now ready to complete the proof 
of Theorem~\ref{thm:r-EO}. 
The operation we wish to apply to 
(\ref{eq:lhs-final}), (\ref{eq:LF-changes-in-2nd-line}),
and (\ref{eq:LF-changes-in-3rd-line})
is 
$$
\left(
d_{\eta_2}\cdots d_{\eta_n}
\right)
\circ
\lcb\bullet\rcb_{\eta_1}.
$$
This means we first calculate the 
principal part of the quantities with respect to
$\eta_1$, and then 
apply the exterior differentiations with respect
to $\eta_2,\dots,\eta_n$. 
We obtain 
\begin{multline}
  \label{eq:LF-final-form}
    \sum_{\substack{\lf|\vk\rt|\equiv0\\ \lf|\vd\rt|\le3g-3+n}}
    r^{\frac{\lf|\vk\rt|}{r}}\,
    \la \tau_{\vd}\,\Lam \ra^{(r),\vk}\,
    d\xi_{\vd}^{r,\vk}(\eta_1,\dots,\eta_n)\\
    =
    \sum_{1<j}
    \sum_{\substack{
        a+\lf|\kWithoutOneJ\rt|\equiv0\\
        m+\lf|\dWithoutOneJ\rt|\le3g-4+n}}
    r^{\frac{a+\lf|\kWithoutOneJ\rt|}{r}}
    \la
    \tau_m\,
    \tau_{\dWithoutOneJ}\,
    \Lam
    \ra^{(r),(a,\kWithoutOneJ)}\,
    R_m^{r,a}(\eta_1,\eta_j)\,
    \otimes\,
    d\xi_{\dWithoutOneJ}^{r,\kWithoutOneJ}
    (\etaWithoutOneJ)\\
    +
    \sum_{\substack{a+b+\lf|\kWithoutOne\rt|\equiv0\\ m+\ell+\lf|\dWithoutOne\rt|\le 3g-5+n}}
    r^{\frac{a+b+\lf|\kWithoutOne\rt|}{r}}\,
    \la\tau_m\,\tau_\ell\,\tau_{\dWithoutOne}\,\Lam
    \ra^{(r),(a,b,\kWithoutOne)}\,
    R_{m,\ell}^{r,a,b}(\eta_1)
    \otimes\,
    d\xi_{\dWithoutOne}^{r,\kWithoutOne}(\etaWithoutOne)\\
    +   
    \sum_{\substack{g_1+g_2=g\\I\sqcup J=[\hat1]}}^{\text{stable}}
    \sum_{\substack{
        a+\lf|k_I\rt|\equiv0\\
        b+\lf|k_J\rt|\equiv0\\
        m+\lf|\ell_I\rt|\le 3g_1-2+|I|\\
        \ell+\lf|\ell_J\rt|\le 3g_2-2+|J|}}
    r^{\frac{a+b+\lf|\kWithoutOne\rt|}{r}}\,
    \la\tau_m\,\tau_{\ell_I}\,\Lam \ra^{(r),(a,k_I)}\,
    \la\tau_\ell\,\tau_{\ell_J}\,\Lam\ra^{(r),(a,k_J)}\,
    \\
    \times
    R_{m,\ell}^{r,a,b}(\eta_1)\otimes
    d\xi_{\dWithoutOne}^{r,\kWithoutOne}(\etaWithoutOne),
\end{multline}
where $d\xi_{\ell_I}^{r,k_I}(\eta_I)=\bigotimes_{i\in
  I}d\xi_{\ell_i}^{r,k_i}(\eta_i)$,
\begin{equation}
  \label{eq:R-m-l-r-a-b}
  R_{m,\ell}^{r,a,b}(\eta_1) :=
  \lcb
  \frac{
    Y^a(\eta_1)\,
    Y^b(\eta_1)\,    
    \phi_{m+1}^{r,a}(v_1)\,
    \phi_{\ell+1}^{r,b}(v_1)\,
    v_1\,dv_1}
  {2\phi_{-1}^{r,0}(v_1)}\rcb_{\eta_1},
\end{equation}
and
\begin{equation}
  \label{eq:R-m-r-a}
  R_m^{r,a}(\eta_1,\eta_j):= -
  d_{\eta_j}
  \lcb
  \frac{\xi_{m+1}^{r,a}(\eta_1)\,d\eta_1}{2\,(\eta_1-\eta_j)\,\phi_{-1}^{r,0}(v_1)}
  +
  \frac{\xi_{m+1}^{r,a}(\etatil_1)\,d\etatil_1}{2\,(\etatil_1-\eta_j)\,\phi_{-1}^{r,0}(v_1)}
  \rcb_{\eta_1}.
\end{equation}

\begin{rem}
  The operation of $\lcb\bullet\rcb_{\eta_1}$ in (\ref{eq:R-m-r-a}) and (\ref{eq:R-m-l-r-a-b}) are
  well defined because of the fact that $\etatil_1$ is holomorphic in $\eta_1\in U$,
  (\ref{eq:v-in-Delta}), Proposition~\ref{prop:analytic-properties-phi-h}, and
  (\ref{eq:xi-decompose-in-phi-h}).
\end{rem}

To deduce (\ref{eq:r-EO}) from 
(\ref{eq:LF-final-form}), 
we need the following  formulas.
\begin{prop}
  \begin{align}
    \label{eq:R-m-l-r-a-b-residue}
    R_{m,\ell}^{r,a,b}(\eta_1) 
    &= 
    r\,
    \sum_{j=1} ^r\Res_{\eta=\a_j}
    K_j(\eta,\eta_1)\,
    d\xi_m^{r,a}(\eta)\,
    \otimes\,
    d\xi_\ell^{r,b}(\etatil)
    \\
    \label{eq:R-m-r-a-residue}
    R_m^{r,a}(\eta_1,\eta_i) 
    &= 
    \sum_{j=1}^r \Res_{\eta=\a_j}
    \bigg[K_j(\eta,\eta_1)
    \bigg(
    B(\eta,\eta_i)\, d\xi_{m}^{r,a}(\etatil) +
    B(\etatil,\eta_i)\, d\xi_{m}^{r,a}(\eta)
    \bigg)
    \bigg],
  \end{align}
  where $B(\eta,\eta_i)$ is defined in (\ref{eq:B}).
\end{prop}

\begin{proof}
  Let $s_j(\gam_j)$ be the involution image
  of a
  small circle $\gam_j$
  around $\a_j$. We calculate the residue by  contour integration. 
  \begin{multline*}
    \text{R.H.S. of (\ref{eq:R-m-l-r-a-b-residue})}=r\,
    \sum_{j=1} ^r\Res_{\eta=\a_j}
    K_j(\eta,\eta_1)\,
    d\xi_m^{r,a}(\eta)\,
    \otimes\,
    d\xi_\ell^{r,b}(\etatil)
     \\
    = 
    \frac{r^2\,d\eta_1}2 
    \sum_{j}
    \frac1{2\pi i}
    \lf[
    \oint_{\gam_j} 
    \frac{\xi_{m+1}^{r,a}(y)}{(y-\eta_1)(\ytilde^r-y^r)}
    \frac{d\xi_{\ell}^{r,b}(\ytilde)}{d\ytilde}\,d\ytilde
    +\oint_{s_j(\gam_j)}
    \frac{\xi_{m+1}^{r,a}(\ytilde)}{(y-\eta_1)(\ytilde^r-y^r)}
    \frac{d\xi_{\ell}^{r,b}(y)}{dy}\,dy
    \rt]
    \\
    =
    \frac{r^2\,d\eta_1}2 
    \sum_{\a_j}
    \Res_{y=\a_j}
    \frac{
      (1-y^r)
      \lf[
      \xi_{m+1}^{r,a}(y)\,\xi_{\ell+1}^{r,b}(\ytilde)+
      \xi_{m+1}^{r,a}(\ytilde)\,\xi_{\ell+1}^{r,b}(y)
      \rt]
    }{(y-\eta_1)\,y\,(\ytilde^r-y^r)}.
  \end{multline*}
  In the second line we have used the involution to the second contour integral, and in the third
  line we have appealed to  (\ref{eq:diff-relations}).
  Noticing that  
  $$\frac{ (1-y^r) \lf[ \xi_{m+1}^{r,a}(y)\,\xi_{\ell+1}^{r,b}(\ytilde)+
    \xi_{m+1}^{r,a}(\ytilde)\,\xi_{\ell+1}^{r,b}(y) \rt] }{y\,(\ytilde^r-y^r)}
    \in \cM_U,
    $$ 
     we use Lemma \ref{lem:residue-calculation},
  (\ref{eq:xi-decompose-in-phi-h}), 
  and (\ref{eq:phi-1}) 
     to yield
  \begin{align*}
    \text{R.H.S.}&= -\frac{r^2}{2}
    \lcb
    \frac{
      (1-\eta_1^r)
      \lf[
      \xi_{m+1}^{r,a}(\eta_1)\,\xi_{\ell+1}^{r,b}(\etatil_1)+
      \xi_{m+1}^{r,a}(\etatil_1)\,\xi_{\ell+1}^{r,b}(\eta_1)
      \rt]
    }{\eta_1\,(\etatil_1^r-\eta_1^r)}
    \rcb_{\eta_1}\,d\eta_1\\
    &=
    \lcb
    \frac{
      Y^a(\eta_1)\,
      Y^b(\eta_1)\,
      \lf[
      \phi_{m+1}^{r,a}(v_1)\,
      \phi_{\ell+1}^{r,b}(v_1)
      -h_{m+1}^{r,a}(w_1)\,
      h_{\ell+1}^{r,b}(w_1)
      \rt]
      v_1\,dv_1}
    {2\phi_{-1}^{r,0}(v_1)}\rcb_{\eta_1}\\
    &=R_{m,\ell}^{r,a,b}(\eta_1).
  \end{align*}

Similarly, for $\eta_i\notin U$, we have
  \begin{align*}
    &\text{R.H.S of (\ref{eq:R-m-r-a-residue})}
    =\sum_{j=1}^r \Res_{\eta=\a_j}
    \bigg[K_j(\eta,\eta_1)
    \bigg(
    B(\eta,\eta_i)\, d\xi_{m}^{r,a}(\etatil) +
    B(\etatil,\eta_i)\, d\xi_{m}^{r,a}(\eta)
    \bigg)
    \bigg]
    \\
    &=
    \frac{r}{2}\,d\eta_1
    \otimes
    d_{\eta_i}
    \lf[
    \sum_{\a_j}
    \frac{1}{2\pi i}
    \oint_{\gamma_j}
    \lf(
    \frac{1}{y-\eta_1}-
    \frac{1}{\ytilde-\eta_1}
    \rt)
    \lf(
    \frac{\xi_{m+1}^{r,a}(\ytilde)}{y-\eta_i} +
    \frac{\xi_{m+1}^{r,a}(y)\,\ytilde'}{\ytilde-\eta_i} 
    \rt)\,
    \frac{dy}{\ytilde^r-y^r}
    \rt]\\
    &=
    r\,d\eta_1\otimes
    d_{\eta_i}
    \lf[
    \sum_{\a_j}
    \Res_{y=\a_j}
    \frac1{y-\eta_1}
    \lf(
    \frac{\xi_{m+1}^{r,a}(\ytilde)}{(y-\eta_i)(\ytilde^r-y^r)}+
    \frac{\xi_{m+1}^{r,a}(y)\,\ytilde'}{(\ytilde-\eta_i)(\ytilde^r-y^r)}
    \rt)
    \rt]\\
    &=
    -r\,d\eta_1\otimes\,d_{\eta_i}
    \lcb
    \frac{\xi_{m+1}^{r,a}(\etatil_1)}{(\eta_1-\eta_i)(\etatil_1^r-\eta_1^r)}+
    \frac{\xi_{m+1}^{r,a}(\eta_1)\,\etatil_1'}{(\etatil_1-\eta_i)(\etatil_1^r-\eta_1^r)}
    \rcb_{\eta_1}\\
    &=d\eta_1\otimes\,d_{\eta_i}
    \lcb
    \frac{\xi_{m+1}^{r,a}(\etatil_1)}{2\,(\eta_1-\eta_i)\,\phi_{-1}^{r,0}(v_1)}+
    \frac{\xi_{m+1}^{r,a}(\eta_1)\,\etatil_1'}{2\,(\etatil_1-\eta_i)\,\phi_{-1}^{r,0}(v_1)}
    \rcb_{\eta_1},
  \end{align*}
 thanks to Lemma~\ref{lem:residue-calculation}. 
 Here the sign $'$ indicates differentiation 
 with respect to the variable without the 
 $\tilde{\;}$-sign. 
 For example, $\ytilde' =   d\ytilde /dy$, etc.
  The last step is to equate the
  above result with (\ref{eq:R-m-r-a}),
which follows from 
  \begin{lem}
  \begin{displaymath}
  \left\{   
    \lf(
    \frac1{\eta_1-\eta_i}+
    \frac{\etatil_1'}{\etatil_1-\eta_i}
    \rt)
    \lf(
    \frac{\xi_{m+1}^{r,a}(\eta_1)+
      \xi_{m+1}^{r,a}(\etatil_1)}{2\phi_{-1}^{r,0}(v_1)}
    \rt)
    \right\}_{\eta_1} = 0.
  \end{displaymath}
  \end{lem}
  \begin{proof}[Proof of Lemma]
 Note that
  \begin{displaymath}
    \frac{\xi_{m+1}^{r,a}(\eta_1)+
      \xi_{m+1}^{r,a}(\etatil_1)}{2\phi_{-1}^{r,0}(v_1)}
    =-\frac{r\,Y^a(\eta_1)\, h_{m+1}^{r,a}(0)}{\Dt_1} + \cO(\Dt_1),
  \end{displaymath}
 which has a simple pole at each $\a_j$. 
 Since $\etatil_1'\big|_{\eta_1=\a_i}=-1$,  the holomorphic
  function $\frac1{\eta_1-\eta_i}+\frac{\etatil_1'}{\etatil_1-\eta_i}$ has a zero at each $\a_j$. Therefore,
  the principal part operation 
  is applied to  a holomorphic function in $\eta_1$,
  hence the result is $0$.
  \end{proof}
  We have now completed the proof of 
  Theorem~\ref{thm:r-EO}.
\end{proof}

\begin{ack}
The authors thank the American Institute of
Mathematics in Palo Alto, the Banff International 
Research Station in Canada, and the Hausdorff 
Research Institute for 
Mathematics in Bonn for their 
support and hospitality, where this collaboration
was started. They also thank 
 Bertrand Eynard,
  Takashi Kimura,
  Renzo Cavalieri,
  Dustin Ross,
  Sergey Shadrin,
  Loek Spitz,
  and 
 Piotr Su\l kowski
 for useful discussions.

 The research of V.B.\ is
 supported by an NSERC Discovery Grant.
 
 D.H.S.\ received research support from research contract MTM2009-11393 of  Ministerio de Ciencia e Innovaci\'on, Jos\'e Castillejo fellowship of  Ministerio de Educaci\'on, and  Becas de movilidad of JCYL, which allowed him
 to conduct research at the 
 Department of Mathematics, University of 
 California, Davis.

 X.L.\ received the  
 China Scholarship Council 
 grant CSC-2010811063, which allowed him
 to conduct research at the 
 Department of Mathematics, University of 
 California, Davis. 
 He is also supported by the National Science 
 Foundation of China grants No.11201477, 11171175, and 10901090, and the
 Chinese Universities Scientific Fund No.2011JS041.
 
The research of M.M.\ has been 
supported by NSF grants
DMS-1104734 and DMS-1104751,
Max-Planck Institut f\"ur Mathematik in Bonn,
 the Beijing International Center
for Mathematical Research,  University of
Salamanca, Universiteit van Amsterdam, 
and the Kavli Institute for
the Physics and Mathematics of the Universe,
Kashiwa.

\end{ack}


\providecommand{\bysame}{\leavevmode\hbox to3em{\hrulefill}\thinspace}

\bibliographystyle{amsplain}

\end{document}